\documentclass{article}

\usepackage[final]{neurips_2023}

\usepackage[utf8]{inputenc} 
\usepackage[T1]{fontenc}    
\usepackage{hyperref}       
\usepackage{xurl}            
\usepackage{booktabs}       
\usepackage{amsfonts}       
\usepackage{nicefrac}       
\usepackage{microtype}      
\usepackage{xcolor}         
\usepackage{amsfonts}
\usepackage{xspace}
\usepackage{dsfont}
\usepackage{hyperref}
\usepackage{nicefrac}
\usepackage{bm}
\usepackage{times}
\usepackage{graphicx}
\usepackage{amsmath,amssymb,amsthm}
\usepackage{multirow}

\newcommand{\real}[0]{\ensuremath{\mathbb{R}}\xspace}                               

\newcommand{\expectation}[0]{\ensuremath{\mathbb{E}}}                                

\newcommand{\covariance}{\ensuremath{\Sigma}\xspace}                                
\newcommand{\bernoulli}[0]{\ensuremath{\delta}\xspace}                              

\newcommand{\diag}[0]{\text{diag}\xspace}                                           
\newcommand{\scalarproduct}[2]{\langle #1 , #2 \rangle}                             
\newcommand{\norm}[1]{\left\Vert #1 \right\Vert}                                     
\newcommand{\psinorm}[2]{\norm{#1}_{\psi_{#2}}}                                     
\newcommand{\trace}[1]{\text{tr}\left(#1\right)}                                                   

\newcommand{\ubar}[1]{\text{\b{$#1$}}}

\newtheorem{theorem}{Theorem}
\newtheorem{definition}{Definition}
\newtheorem{proposition}{Proposition}
\newtheorem{lemma}{Lemma}

\title{Robust covariance estimation with missing values and cell-wise contamination}

\author{
  Karim Lounici
  \\
  CMAP\\
  Ecole Polytechnique\\
  Palaiseau, France \\
karim.lounici@polytechnique.edu
\And
   Gregoire Pacreau
\\
  CMAP\\
  Ecole Polytechnique\\
  Palaiseau, France\\
gregoire.pacreau@polytechnique.edu
}

\begin{document}

\maketitle

\begin{abstract}
Large datasets are often affected by cell-wise outliers in the form of missing or erroneous data. However, discarding any samples containing outliers may result in a dataset that is too small to accurately estimate the covariance matrix. 
Moreover, the robust procedures designed to address this problem require the invertibility of the covariance operator and thus are not effective on high-dimensional data. In this paper, we propose an unbiased estimator for the covariance in the presence of missing values that does not require any imputation step and still achieves near minimax statistical accuracy with the operator norm. We also advocate for its use in combination with cell-wise outlier detection methods to tackle cell-wise contamination in a high-dimensional and low-rank setting, where state-of-the-art methods may suffer from numerical instability and long computation times. To complement our theoretical findings, we conducted an experimental study which demonstrates the superiority of our approach over the state of the art both in low and high dimension settings.
\end{abstract}

\section{Introduction}

Outliers are a common occurrence in datasets, and they can significantly affect the accuracy of data analysis. While research on outlier detection and treatment has been ongoing since the 1960s, much of it has focused on cases where entire samples are outliers (Huber's contamination model) \citep{huberRobustEstimationLocation1964, tukeyNintherTechniqueLowEffort1978, hubertMinimumCovarianceDeterminant2018}.
While sample-wise contamination is a common issue in many datasets, modern data analysis often involves combining data from multiple sources. For example, data may be collected from an array of sensors, each with an independent probability of failure, or financial data may come from multiple companies, where reporting errors from one source do not necessarily impact the validity of the information from the other sources. Discarding an entire sample as an outlier when only a few features are contaminated can result in the loss of valuable information, especially in high-dimensional datasets where samples are already scarce. It is important to identify and address the specific contaminated features, rather than simply treating the entire sample as an outlier.
In fact, if each dimension of a sample has a contamination probability of $\varepsilon$, then the probability of that sample containing at least one outlier is given by $1-(1-\varepsilon)^p$, where $p$ is the dimensionality of the sample. In high dimension, this probability can quickly exceed $50\%$, surpassing the breakdown point of many robust estimators designed for the Huber sample-wise contamination setting. Hence, it is crucial to develop robust methods that can handle cell-wise contaminations and still provide accurate results.

The issue of cell-wise contamination, where individual cells in a dataset may be contaminated, was first introduced in \citep{alqallafPropagationOutliersMultivariate2009}. However, the issue of missing data due to outliers was studied much earlier, dating back to the work of \cite{rubinInferenceMissingData1976}. Although missing values in a dataset are much easier to detect than outliers, they can 
lead to errors in estimating the location and scale of the underlying distribution \citep{littleStatisticalAnalysisMissing2002} and can negatively affect the performance of supervised learning algorithms \citep{josseConsistencySupervisedLearning2020}. This motivated the development of the field of data imputation. 
Several robust estimation methods have been proposed to handle missing data, including Expectation Maximization (EM)-based algorithms \citep{dempsterMaximumLikelihoodIncomplete1977}, maximum likelihood estimation \citep{jamshidianMLEstimationMean1999} and Multiple Imputation \citep{littleStatisticalAnalysisMissing2002}, among which we can find k-nearest neighbor imputation \cite{troyanskayaMissingValueEstimation2001} and iterative imputation \cite{buurenMiceMultivariateImputation2011}.
Recently, sophisticated solutions based on deep learning, GANs \cite{yoonGAINMissingData2018, matteiMIWAEDeepGenerative2019,dongGenerativeAdversarialNetworks2021}, VAE \cite{maVAEMDeepGenerative2020} or Diffusion schemes \cite{zhengDiffusionModelsMissing2023} have been proposed to perform complex tasks like artificial data generation or image inpainting. The aforementioned references focus solely on 
minimising the entrywise error for imputed entries. Noticeably, our practical findings reveal that applying state-of-the-art imputation methods to complete the dataset, followed by covariance estimation on the completed dataset, does not yield satisfactory results when evaluating the covariance estimation error using the operator norm.

\begin{figure}
    \begin{minipage}{0.5\textwidth}
    \centering
    \includegraphics[width=\textwidth]{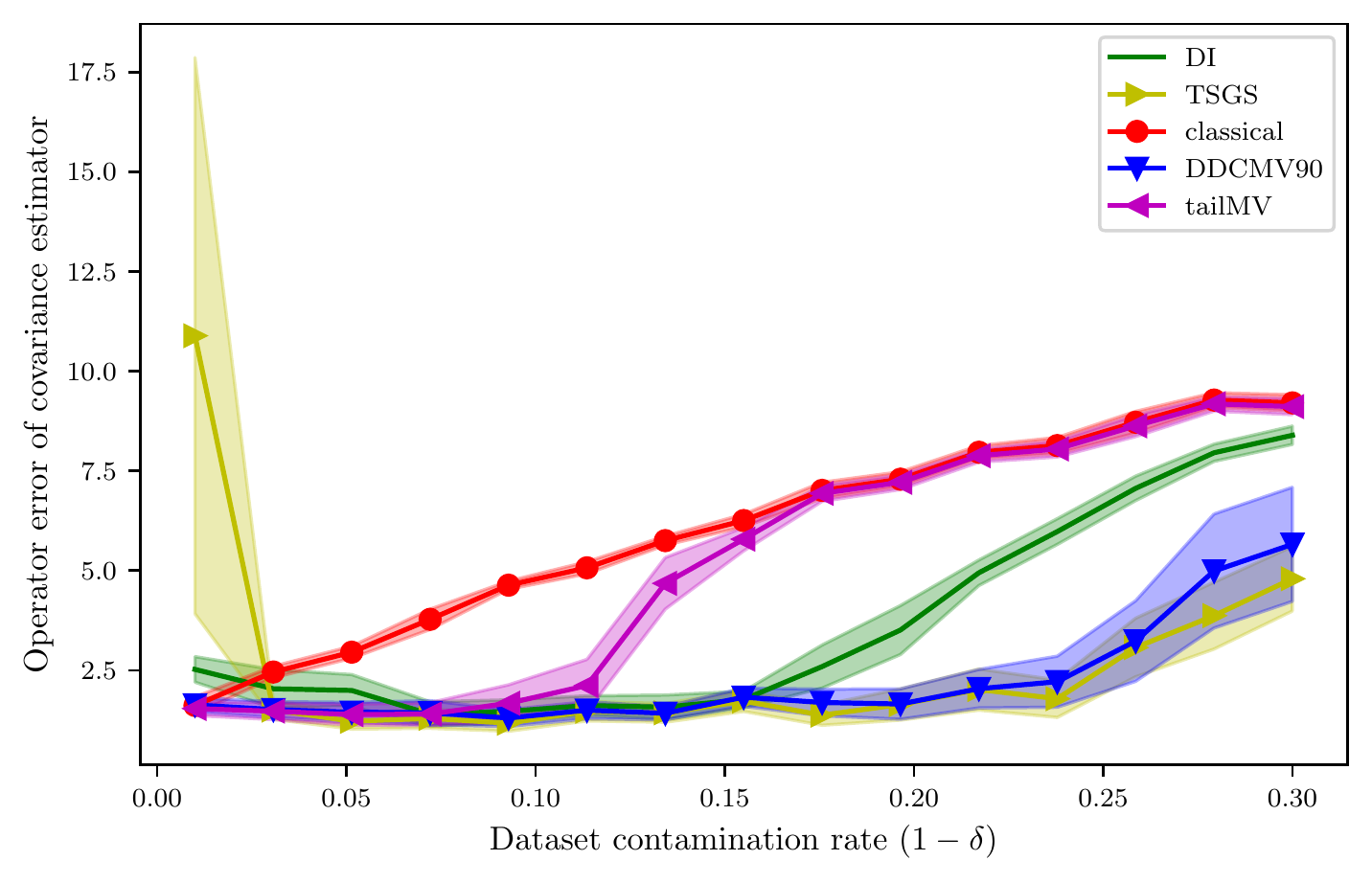}
    \label{fig:low_dim_estim}
    \end{minipage}
    \begin{minipage}{0.3\textwidth}
    \scalebox{0.8}{
    \begin{tabular}{c|ccc}
        \toprule
        \multirow{2}{*}{Estimator} & \multicolumn{2}{c}{Computation time} & Matrix \\
        & p=50 & p=100 & inversion\\
        \midrule
        tailMV & $10^{-3} \pm 10^{-4}$ & $10^{-2} \pm 10^{-3}$  & no \\
        DDCMV &  $0.6 \pm 10^{-3}$ & $0.7\pm 0.007$ & no \\
        DI & $6 \pm 0.5$ & $ 74\pm 4$ & yes\\
        TSGS & $20 \pm 0.2$ & $ 200 \pm 10$ & yes\\
        \bottomrule
    \end{tabular}}
    \end{minipage}
    \caption{Left: Estimation error of the covariance matrix for $n=100$, $p=50$, $\mathbf{r}(\Sigma)=2$ under a Dirac contamination (\texttt{tailMV} and \texttt{DDCMV} are our methods). Here $\varepsilon=1$ and $\delta$ varies in $(0,1)$. Right: For each method, mean computation time (in seconds) over 20 repetitions and whether it uses matrix inversion. For $p=100$, we had to raise $r\left(\Sigma\right)$ to $10$ otherwise both \texttt{DI} and \texttt{TSGS} would fail due to numerical instability.}
\end{figure}
    
In comparison to data missingness or its sample-wise counterpart, the cell-wise contamination problem is less studied.
The Detection Imputation (DI)  algorithm of \cite{raymaekersHandlingCellwiseOutliers2020} is an EM type procedure combining a robust covariance estimation method with an outlier detection method to iteratively update the covariance estimation. 
Other methods include adapting methodology created for Huber contamination for the cell-wise problem, such as in \cite{danilovRobustEstimationMultivariate2012} or \cite{agostinelliRobustEstimationMultivariate2014}. In high dimensional statistics, however, most of these methods fail due to high computation time and numerical instability. Or they are simply not designed to work in this regime since they are based on the Mahalanobis distance, which requires an inversion of the estimated covariance matrix. This is a major issue since classical covariance matrix estimators have many eigenvalues close to zero or even exactly equal to zero in high-dimension. To the best of our knowledge, no theoretical result exists concerning the statistical accuracy of these methods in the cell-wise contamination setting contrarily to the extensive literature on Huber's contamination \cite{abdallaCovarianceEstimationOptimal2023}.

\textbf{Contributions.} In this paper we address the problem of high-dimensional covariance estimation in the presence of missing observations and cell-wise contamination. 
To formalize this problem, we adopt and generalize the setting introduced in \cite{farcomeniRobustConstrainedClustering2014}. We propose and investigate two different strategies, the first based on filtering outliers and debiasing and the second based on filtering outliers followed by imputation and standard covariance estimation. We propose novel computationally efficient and numerically stable procedures that avoid matrix inversion, making them well-suited for high-dimensional data. We derive non-asymptotic estimation bounds of the covariance with the operator norm and 
minimax lower bounds, which clarify the impact of the missing value rate and outlier contamination rate. Our theoretical results also improve over \cite{louniciHighdimensionalCovarianceMatrix2014} in the MCAR and no contamination. Next, we conduct an experimental study on synthetic data, comparing our proposed methods to the state-of-the-art (SOTA) methods. Our results demonstrate that SOTA methods fail in the high-dimensional regime due to matrix inversions, while our proposed methods perform well in this regime, highlighting their effectiveness. Then we demonstrate the practical utility of our approach by applying it to real-life datasets, which highlights that the use of existing estimation methods significantly alters the spectral properties of the estimated covariance matrices. This implies that cell-wise contamination can significantly impact the results of dimension reduction techniques like PCA by completely altering the computed principal directions. Our experiments demonstrate that our methods are more robust to cell-wise contamination than SOTA methods and produce reliable estimates of the covariance.

\section{Missing values and cell-wise contamination setting}

    Let $X_1, \dots, X_n$ be $n$ i.i.d. copies of a zero mean random vector $X$ admitting unknown covariance operator $\covariance = \expectation\left[ X \otimes X \right]$, where $\otimes$ is the outer product. Denote by $X_i^{(j)}$ the $j$th component of vector $X_i$ for any $j\in [p]$. All our results are non-asymptotic and cover 
    a wide range of configurations for $n$ and $p$ including the high-dimensional setting $p\gg n$. In this paper, we consider the following two realistic scenarios where the measurements are potentially corrupted.

\paragraph{Missing values.} We assume that each component $X_i^{(j)}$ is observed independently from the others with probability $\bernoulli \in (0,1]$. Formally, we observe the random vector $Y\in \mathbb{R}^p$ defined as follows:
\begin{equation}
        \label{eqn:missing_values}
        Y_i^{(j)} = d_{i,j}X_i^{(j)}, 1 \leq i \leq n, 1\leq j\leq p
\end{equation}
where $d_{ij}$ are independent realisations of a bernoulli random variable of parameter $\delta$. This corresponds the Missing Completely at Random (MCAR) setting of \cite{rubinInferenceMissingData1976}. Our theory also covers the more general Missing at Random (MAR) setting in Theorem \ref{thm:MAR}.

\paragraph{Cell-wise contamination.}
Here we assume that some missing components $X_i^{(j)}$ can be replaced with probability $\varepsilon$ by some independent noise variables, representing either a poisoning of the data or random mistakes in measurements. The observation vector $Y$ then satisfies:
    \begin{equation}
        \label{eqn:contaminated}
        Y_i^{(j)} = d_{i,j}X_i^{(j)} + (1-d_{i,j})e_{i,j}\xi_i^{(j)}, 1 \leq i \leq n, 1\leq j\leq p
    \end{equation}
    where $\xi_i^{(j)}$
    are independent erroneous measurements and $e_{i,j}$ are i.i.d. bernoulli random variables with parameter $\varepsilon$. We also assume that all the variables $X_i$, $\xi_i^{(j)}$, $d_{i,j}$, $e_{i,j}$ are mutually independent. In this scenario, a component $X_i^{(j)}$ is either perfectly observed with probability $\delta$, replaced by a random noise with probability $\varepsilon' = \varepsilon(1-\delta)$ or missing with probability $(1-\delta)(1-\varepsilon)$. 
Cell-wise contamination as introduced in \cite{alqallafPropagationOutliersMultivariate2009} corresponds to the case where $\varepsilon = 1$, and thus $\varepsilon' = 1-\delta$.

    In both of these settings, the task of estimating the mean of the random vectors $X_i$ is well-understood, as it reduces to the classical Huber setting for component-wise mean estimation. One could for instance apply the Tuker median on each component separately \cite{alqallafPropagationOutliersMultivariate2009}. However, the problem becomes more complex when we consider non-linear functions of the data, such as the covariance operator. Robust covariance estimators originally designed for the Huber setting may not be suitable when applied in the presence of missing values or cell-wise contaminations.

    We study a simple estimator based on a correction of the classical covariance estimator on $Y_1, \dots, Y_n$ as introduced in \cite{louniciHighdimensionalCovarianceMatrix2014} for the missing values scenario. The procedure is based on the following observation, linking $\Sigma^Y$ the covariance of the data with missing values and $\Sigma$ the true covariance:
    \begin{equation}
    \label{eqn:mvcorrection}
    \Sigma = \left( \delta^{-1} - \delta^{-2}\right) \diag (\Sigma^Y) + \delta^{-2} \Sigma^Y
    \end{equation}
    Note that this formula assumes the knowledge of $\delta$. In the missing values scenario, $\delta$ can be efficiently estimated by a simple count of the values exactly set to $0$ or equal to \texttt{NaN} (not a number). 
    In the contamination setting \eqref{eqn:contaminated}, the operator $\Sigma^Y = \expectation \left( Y \otimes Y \right)$ satisfies, for $\Lambda = \expectation \left[ \xi \otimes \xi \right]$:
    $$
    \Sigma^Y = \delta^2 \Sigma + (\delta-\delta^2) \mathrm{diag}(\Sigma) + \varepsilon(1-\delta) \Lambda.
    $$
    In this setting, as one does not know the exact location and number of outliers we propose to estimate $\delta$ by the proportion of data remaining after the application of a filtering procedure.

\paragraph{Notations.} We denote by $\odot$ the Hadamard (or term by term) product of two matrices and by $\otimes$ the outer product of vectors, i.e. $\forall x,y \in \real^d, x \otimes y = x y^\top$. We denote by $\norm{.}$ and $\norm{.}_F$ the operator and Frobenius norms of a matrix respectively. We denote by $\norm{\cdot}_2$ the vector $l_2$-norm.

\section{Estimation of covariance matrices with missing values}
\label{sec:missing_values}
We consider the scenario outlined in \eqref{eqn:missing_values} where the matrix $\Sigma$ is of approximately low rank. To quantify this, we use the concept of effective rank, which provides a useful measure of the inherent complexity of a matrix. Specifically, the effective rank of $\Sigma$ is defined as follows
    \begin{equation}
        \bm{r}(\Sigma) := \frac{\expectation \norm{X}_2^2 }{\norm{\Sigma}} = \frac{\trace{\Sigma}}{\norm{\Sigma}}
    \end{equation}
We note that $0 \leq \bm{r}(\Sigma) \leq \text{rank}(\Sigma)$. Furthermore, for approximately low rank matrices with rapidly decaying eigenvalues, we have $\bm{r}(\Sigma) \ll \text{rank}(\Sigma)$. This section presents a novel analysis of the estimator defined in equation \eqref{eqn:mvcorrection}, which yields a non-asymptotic minimax optimal estimation bound in the operator norm.  Our findings represent a substantial enhancement over the suboptimal guarantees reported in \cite{louniciHighdimensionalCovarianceMatrix2014, klochkovUniformHansonWrightType2019}.
Similar results could be established for the Frobenius norm using more straightforward arguments, as those in \cite{buneaSampleCovarianceMatrix2015} or \cite{puchkinSharperDimensionfreeBounds2023}. We give priority to the operator norm since it aligns naturally with learning tasks such as PCA. See \cite{15-AIHP705,Koltchinskii2017,16-AOS1437} and the references cited therein.

We need the notion of Orlicz norms. For any $\alpha \geq 1$, the $\psi_\alpha$-norms of a real-valued random variable $V$ are defined as: $\psinorm{V}{\alpha} = \inf \lbrace u> 0, \expectation\exp\left(\vert V \vert^\alpha / u^\alpha \right) \leq 2 \rbrace $.  A random vector $X \in \real^p$ is sub-Gaussian if and only if $\forall x \in \real^p$, $\psinorm{\scalarproduct{X}{x}}{2} \lesssim \norm{\scalarproduct{X}{x}}_{L^2}$.


    \paragraph{Minimax lower-bound.}
        
        We now provide a minimax lower bound for the covariance estimation with missing values problem. Let $\mathcal{S}_p$ the set of $p\times p$ symmetric semi-positive matrices. Then, define $\mathcal{C}_{\overline{r}} = \lbrace S\in \mathcal S_p : \bm{r}(S) \leq \overline{r} \rbrace$ the set of matrices of $\mathcal{S}_p$ with effective rank at most $\overline{r}$.
        \begin{theorem}
        \label{th:lower}
        Let $p,n, \overline{r}$ be strictly positive integers such that $p\geq  \max \{ n, 2 \overline{r}\} $. 
        Let $X_1, \dots, X_n$ be i.i.d. random vectors in $\mathbb{R}^p$ with covariance matrix $\Sigma \in \mathcal{C}_{\overline{r}}$. Let $(d_{i,j})_{1\leq i \leq n, 1\leq j\leq p}$ be an i.i.d. sequence of Bernoulli random variables with probability of success $\delta\in (0,1]$, independent from the $X_1,\dots, X_n$. We observe $n$ i.i.d. vectors $Y_1, \dots, Y_n \in \mathbb{R}^p$ such that $Y_i^{(j)} = d_{i,j} X_i^{(j)}$, $i\in [n]$, $j\in [p]$. Then there exists two absolute constants $C > 0$ and $\beta \in (0,1)$ such that:
        \begin{equation}
            \inf_{\widehat{\Sigma}} \max_{\Sigma\in \mathcal{C}_{\overline{r} }} \mathbb{P}_\Sigma \left(\norm{\widehat{\Sigma} - \Sigma} \geq C \frac{\norm{\Sigma}}{\delta}\sqrt{\frac{\bm{r}(\Sigma)}{n}} \right) \geq \beta
        \end{equation}
        where $\inf_{\widehat{\Sigma}}$ represents the infimum over all estimators $\widehat{\Sigma}$ of matrix $\Sigma$ based on $Y_1, \dots, Y_n$.
        \end{theorem}
        \begin{proof}[Sketch of proof] We first build a sufficiently large test set of hard-to-learn covariance operators exploiting entropy properties of the Grassmann manifold such that the distance between any two distinct covariance operator is at least of the order $\frac{\norm{\Sigma}}{\delta}\sqrt{\frac{\bm{r}(\Sigma)}{n}}$. Next, in order to control the Kullback-Leibler divergence of the observations with missing values, we exploit in particular interlacing properties of the eigenvalues of the perturbed covariance operators \cite{thompsonPrincipalSubmatricesNormal1966}.
        \end{proof}
        
        This lower bound result improves upon \citep[Theorem 2]{louniciHighdimensionalCovarianceMatrix2014} as it 
        relaxes the hypotheses on $n$ and $\overline{r}$. More specifically, the lower bound in \cite{louniciHighdimensionalCovarianceMatrix2014} requires $n \geq 2\overline{r}^2/\delta^2$ while we only need the mild assumption $p \geq \max\{n, 2\overline{r}\}$.
Our proof 
leverages the properties of the Grassmann manifold, which has been previously utilized in different settings such as sparse PCA without missing values or contamination \cite{vuMinimaxSparsePrincipal2013} and low-rank covariance estimation without missing values or contamination \cite{koltchinskiiEstimationLowRankCovariance2015}. However, tackling missing values in the Grassmann approach adds a technical challenge to these proofs as they modify the distribution of observations. Our proof requires several additional nontrivial arguments to control the distribution divergences, which is a crucial step in deriving the minimax lower bound.

\paragraph{Non-asymptotic upper-bound in the operator norm.}

    We provide an upper bound of the estimation error in operator norm. We write $Y_i = d_i \odot X_i$. Let $\widehat{\Sigma}^Y = n^{-1}\sum_{i=1}^n Y_i \otimes Y_i$ be the classical covariance estimator of the covariance of $Y$. When the dataset contains missing values and corruptions, $\widehat{\Sigma}^Y$ is a biased estimator of $\Sigma$. Exploiting Equation \eqref{eqn:mvcorrection}, \cite{louniciHighdimensionalCovarianceMatrix2014} proposed the following unbiased estimator of the covariance matrix $\Sigma$:
    \begin{equation}
    \label{eq:debiased_Covest}
        \widehat{\Sigma} = \delta^{-2} \widehat{\Sigma}^Y + (\delta^{-1} - \delta^{-2})\text{diag}\left(\widehat{\Sigma}^Y\right).
    \end{equation}

    The following result is from \citep[Theorem 4.2]{klochkovUniformHansonWrightType2019}.
    \begin{lemma}
        \label{th:upper}
        Let $X_1, \dots, X_n$ be i.i.d. sub-Gaussian random variables in $\real^p$, with covariance matrix $\Sigma$, and let $d_{ij}, i \in [1, n], j \in [1,p]$ be i.i.d bernoulli random variables with probability of success $\delta>0$. 
        Then there exists an absolute constant $C$ such that, for $t>0$, with probability at least $1-e^{-t}$:
        \begin{equation}
        \label{eq:mainupperboundmissing}
            \norm{ \widehat{\Sigma} - \Sigma} \leq C \norm{\Sigma}\left(\sqrt{\frac{\bm{r}(\Sigma) \log\bm{r}(\Sigma)}{\delta^2 n}} \lor \sqrt{\frac{t}{\delta^2n}} \lor \frac{\bm{r}(\Sigma)(t+\log\bm{r}(\Sigma))}{ \delta^2 n}\log(n)\right)
        \end{equation}
    \end{lemma}
 This result uses a recent unbounded version of the non-commutative Bernstein inequality, thus yielding some improvement upon the previous best known bound of  
\cite{louniciHighdimensionalCovarianceMatrix2014}.
Theorem \ref{th:lower} and Lemma \ref{th:upper} provide some important insights on the minimax rate of estimation in the missing values setting. In the high-dimensional regime $p\geq \max\{n, 2\overline{r}\}$ and $n\geq \delta^{-2}\mathbf{r}(\Sigma) (\log \mathbf{r}(\Sigma)) \log^2 n$, we observe that the two bounds coincide up to a logarithmic factor in $\mathbf{r}(\Sigma)$, hence clarifying the impact of missing data on the estimation rate via the parameter $\delta$.


        
  \paragraph{Heterogeneous missingness.} We can extend the correction to the more general case where each feature has a different missing value rate known as the Missing at Random (MAR) setting in \cite{rubinInferenceMissingData1976}. We denote by $\delta_j\in (0,1]$ the probability to observe feature $X^{(j)}$, $1\leq j \leq p$ and we set $\delta:= (\delta_j)_{j\in [p]}$. As in the MCAR setting, the probabilities $(\delta_j)_{j\in [p]}$ can be readily estimated by tallying the number of missing entries for each feature. Hence they will be assumed to be known for the sake of brevity. Let $\delta_{\text{inv}} = (\delta_j^{-1})_{j\in [p]}$ be the vector containing the inverse of the observing probabilities and $\Delta_{\text{inv}} = \delta_{\text{inv}} \otimes \delta_{\text{inv}}$. In this case, the corrected estimator becomes :
  \begin{equation}
    \label{eq:corr_hetero}
      \widehat{\Sigma} = \Delta_{\text{inv}} \odot \widehat{\Sigma}^Y + \bigl(\diag\left(\delta_{\text{inv}}\right) - \Delta_{\text{inv}}\bigr)\odot \diag \left(\widehat{\Sigma}^Y\right)
  \end{equation}
Let $\bar{\delta} = \max_j \{\delta_j\}$ and $\ubar{\delta} = \min_j \{\delta_j\}$ be the largest and smallest probabilities to observe a feature.
\begin{theorem}
\label{thm:MAR}
(i) Let $X_1, \dots, X_n$ be i.i.d. sub-Gaussian random variables in $\real^p$, with covariance matrix $\Sigma$. We consider the MAR setting described above. Then the estimator \eqref{eq:corr_hetero} satisfies, for any $t>0$, with probability at least $1-e^{-t}$
\begin{equation}
\label{eq:upper_hetero}
    \norm{ \widehat{\Sigma} - \Sigma} \leq C \norm{\Sigma} \frac{\bar{\delta}}{\ubar{\delta}^2}\left(\sqrt{\frac{\bm{r}(\Sigma) \log\bm{r}(\Sigma)}{n}} \lor \sqrt{\frac{t}{n}} \lor \frac{\bm{r}(\Sigma)(t+\log\bm{r}(\Sigma))}{ \bar{\delta} n}\log n\right)
\end{equation}
(ii) 
Let $p,n, \overline{r}$ be strictly positive integers such that $p\geq  \max \{ n, 2 \overline{r}\} $. 
        Let $X_1, \dots, X_n$ be i.i.d. random vectors in $\mathbb{R}^p$ with covariance matrix $\Sigma \in \mathcal{C}_{\overline{r}}$. Then,
        \begin{equation}
\label{eq:under_hetero}
            \inf_{\widehat{\Sigma}} \max_{\Sigma\in \mathcal{C}_{\overline{r} }} \mathbb{P}_\Sigma \left(\norm{\widehat{\Sigma} - \Sigma} \geq C \frac{\norm{\Sigma}}{\bar{\delta}}\sqrt{\frac{\bm{r}(\Sigma)}{ n}} \right) \geq \beta.
\end{equation}
\end{theorem}
If $\bar{\delta}\asymp \ubar{\delta} $ then the rates for the MCAR and MAR settings match. The proof is a straightforward adaptation of the proof in the MCAR setting. 

\section{Optimal estimation of covariance matrices with cell-wise contamination} 

In this section, we consider the cell-wise contamination setting \eqref{eqn:contaminated}.We derive both an upper bound on the operator norm error of the estimator \eqref{eq:debiased_Covest} and a minimax lower bound for this specific setting. Let us assume that the $\xi_1, \dots \xi_n$ are sub-Gaussian r.v. Note also that $\Lambda := \mathbb{E}[\xi_1\otimes \xi_1]$ is diagonal in the cell-wise contamination setting \eqref{eqn:contaminated}.

\paragraph{Minimax lower-bound.}

    The lower bound for missing values still applies to the contaminated case as missing values are a particular case of cell-wise contamination. But we want a more general lower bound that also covers the case of adversarial contaminations.

   \begin{theorem}
    \label{th:lower_contaminated}
    Let $p,n, \overline{r}$ be strictly positive integers such that $p \geq \max\{ n, 2 \overline{r}\}$. Let $X_1, \dots, X_n$ be i.i.d. random vectors in $\mathbb{R}^p$ with covariance matrix $\Sigma \in \mathcal{C}_{\overline{r}}$. Let $(d_{i,j})_{1\leq i \leq n, 1\leq j\leq p}$ be i.i.d. sequence of bernoulli random variables of probability of success $\delta \in (0, 1]$, independent to the $X_1,\dots, X_n$. We observe $n$ i.i.d. vectors $Y_1, \dots, Y_n \in \mathbb{R}^p$ satisfying \eqref{eqn:contaminated} 
    where $\xi_i$ are i.i.d. of arbitrary distribution $Q$. Then there exists two absolute constants $C > 0$ and $\beta \in (0,1)$ such that:
    \begin{equation}
        \inf_{\widehat{\Sigma}} \max_{\Sigma\in \mathcal{C}_{\overline{r} }} \max_{Q} \mathbb{P}_{\Sigma,Q} \left(\norm{\widehat{\Sigma} - \Sigma} \geq C \frac{\norm{\Sigma}}{\delta}\sqrt{\frac{\bm{r}(\Sigma) }{ n}} \bigvee \frac{\varepsilon(1-\delta)}{\delta} \right) \geq \beta
    \end{equation}
    where $\inf_{\widehat{\Sigma}}$ represents the infimum over all estimators of matrix $\Sigma$ and $\max_Q$ is the maximum over all contamination $Q$.
    \end{theorem}
    The proof of this theorem adapts an argument developed to derive minimax lower bounds in the Huber contamination setting. See App. \ref{sec:proof_lower_bound_contamination} for the full proof.

\paragraph{Non-asymptotic upper-bound in the operator norm.} 

    Note that the term $\varepsilon(1-\delta) \Lambda$ in the cell-wise contamination setting is negligible when $\delta \approx 1$ or $\varepsilon \approx 0$. Using the DDC detection procedure of \cite{raymaekersHandlingCellwiseOutliers2020}, we can detect the contaminations and make $\varepsilon$ smaller without decreasing $\delta$ too much. 
    For simplicity, we assume from now on that the $\xi_i^{(j)}$ are i.i.d. with common variance $\sigma_{\xi}^2$. Hence $\Lambda = \sigma_{\xi}^2 I_p$.  We further assume that the  $\xi_i^{(j)}$ are sub-Gaussian since we observed in our experiments that filtering removed all the large-valued contaminated cells and only a few inconspicuous contaminated cells remained. Our procedure \eqref{eq:debiased_Covest} satisfies the following result.
    \begin{theorem}
        \label{th:upper_contaminated}
        Let the assumptions of Theorem \ref{th:upper} be satisfied. We assume in addition that the observations $Y_1,\ldots,Y_n$ satisfy \eqref{eqn:contaminated} with $\varepsilon \in [0,1)$ and $\delta\in (0,1]$ and i.i.d. sub-Gaussian $\xi_i^{(j)}$'s. Then, for any $t>0$, with probability at least $1-e^{-t}$: 
        \begin{equation*}
        \begin{split}
            \norm{\widehat{\Sigma} - \Sigma} & \lesssim \norm{\Sigma}\left(\sqrt{\frac{\bm{r}(\Sigma) \log\bm{r}(\Sigma)}{\delta^2 n}} \lor \sqrt{\frac{t}{\delta^2n}} \lor \frac{\bm{r}(\Sigma)(t+\log\bm{r}(\Sigma))}{ \delta^2 n}\log(n)\right) + \frac{\varepsilon(1-\delta)\sigma_{\xi}^2 }{\delta} \\
            &\hspace{0.25cm}+\frac{(1-\delta)\varepsilon}{\delta^2\sqrt{|\log ((1-\delta)\varepsilon)}|} \sigma_{\xi}^2 \left( \sqrt{\frac{p}{n}} \vee \frac{p}{n} \vee \sqrt{\frac{t}{n}} \vee \frac{t}{n} \right)\\
            &\hspace{1cm}+D(\delta, p)  \sqrt{\frac{t+\log (p)}{n}} + \sqrt{\delta (1-\delta)\varepsilon\,\sigma_{\xi}^2\, p} \sqrt{ \mathrm{tr}(\Sigma)} \log(n) \frac{t+\log (p)}{n},
        \end{split}
        \end{equation*}
        where $D(\delta, p)=\sqrt{\frac{(1-\delta)}{\delta^{2}}\varepsilon (p-2)\sigma_{\xi}^2 \left[2 \norm{\Sigma} + \sigma_{\xi}^2 \right] + \frac{(1-\delta)}{\delta^{3}} \varepsilon \sigma_{\xi}^4\left( \left| \mathrm{tr}(\Sigma) - \delta (p-2) \right| +\norm{ \Sigma}\right)}$.
        \end{theorem}

    See App \ref{sec:proofcontaminated} for the proof.
    As emphasized in \cite{koltchinskiiConcentrationInequalitiesMoment2017}, the effective rank $\mathbf{r}(\Sigma)$ provides a measure of the statistical complexity of the covariance learning problem in the absence of any contamination. However, when cell-wise contamination is present, the statistical complexity of the problem may increase from $\mathbf{r}(\Sigma)$ to $\mathbf{r}(\Lambda) = p$. Fortunately, if the filtering process reduces the proportion of cell-wise contamination $\varepsilon$ such that $(1-\delta)\varepsilon\,\mathrm{tr}(\Lambda) \leq \delta \mathrm{tr}(\Sigma)$ and $\varepsilon\,\norm{\Lambda}\leq \delta \norm{\Sigma}$. Then we can effectively mitigate the impact of cell-wise contamination. Indeed, we deduce from Theorem \ref{th:upper_contaminated} that
    \begin{equation}
    \label{eq:corcontaminated}
    \begin{split}
        \norm{\widehat{\Sigma} - \Sigma} \lesssim & \norm{\Sigma}\left(\sqrt{\frac{\bm{r}(\Sigma) \log\bm{r}(\Sigma)}{\delta^2 n}} \lor \sqrt{\frac{t}{\delta^2n}} \lor \frac{\bm{r}(\Sigma)(t+\log\bm{r}(\Sigma))}{ \delta^2 n}\log(n)\right) + \frac{\varepsilon(1-\delta)\sigma_{\xi}^2 }{\delta} \\
        &\hspace{0.25cm} + \frac{1}{\delta} \sqrt{(1-\delta)\mathrm{tr}(\Sigma)} \sqrt{\frac{t+\log (p)}{n}}\biggl( \sqrt{\delta \, \sigma_\xi^2} + \sqrt{ \mathrm{tr}(\Sigma)}  \log(n) \sqrt{\frac{t+\log (p)}{n}} \biggr),
    \end{split}
    \end{equation}
    where we considered for convenience the reasonable scenario where $\delta\,(p-2)\geq \mathrm{tr}(\Sigma)$ and $\sigma^2_{\xi}\geq \norm{\Sigma}$. 
    The combination of the upper bound \eqref{eq:corcontaminated} with the lower bound in Theorem \eqref{th:lower_contaminated} provides the first insights into the impact of cell-wise contamination on covariance estimation.  

\section{Experiments}

    In our experiments, $\texttt{MV}$ refers either to the debiased MCAR covariance estimator \eqref{eq:debiased_Covest} or to its MAR extension \eqref{eq:corr_hetero}. The synthetic data generation is described in App. \ref{sec:data_gen}. We also performed experiments on real life datasets described in App. \ref{sec:real_data}. All experiments were conducted on a 2020 MacBook Air with a M1 processor (8 cores, 3.4 GHz). 
    \footnote{Code available at \url{https://github.com/klounici/COVARIANCE_contaminated_data}}

    \subsection{Missing Values}

        We compared our method to popular imputations methods: \texttt{KNNImputer} (\texttt{KNNI}), which imputes the missing values based on the k-nearest neighbours \citep{troyanskayaMissingValueEstimation2001}, and \texttt{IterativeImputer} (\texttt{II}), which is inspired by the \texttt{R} package \texttt{MICE} \citep{buurenMiceMultivariateImputation2011}, as coded in \texttt{sklearn} \cite{scikit-learn}; and two recent GANs-based imputation methods \texttt{MIWAE} \citep{matteiMIWAEDeepGenerative2019} and \texttt{GAIN} \citep{yoonGAINMissingData2018} as found in the package \texttt{hyperimpute} \citep{Jarrett2022HyperImpute}. The deep methods were tested using the same architectures, hyperparameters and early stopping rules as their respective papers. 
        
        In Figures \ref{fig:syn_mv_error}, \ref{fig:gauss4estim-missing} and Table \ref{tab:syn_mv_exec}, we compare our estimator $\texttt{MV}$ defined in \eqref{eq:debiased_Covest} to these imputation methods combined with the usual covariance estimator on synthetic data (see App. \ref{sec:data_gen} for details of data generation) in terms of statistical accuracy and execution time. First, 
         \texttt{MV} beats all other methods in low-dimensional scenarios and maintains a competitive edge with \texttt{II} in high-dimensional situations when the missing data rate remains below $30\%$. Furthermore, it stands as the second-best choice when dealing with missing data rates exceeding $35\%$.
         Next, \texttt{MV} has by far the smallest execution time down several orders of magnitude while the execution time of \texttt{II} increases very quickly with the dimension and can become impractical (see Figure \ref{fig:nasdaq} for a dataset too large for \texttt{II}). Overall, the procedures \texttt{MV} and \texttt{II} perform better than \texttt{MIWAE} and \texttt{GAIN} in this experiment. Our understanding is that \texttt{MIWAE} and \texttt{GAIN} use training metrics designed to minimize the entrywise error of imputation. We suspect this may be why their performances for the estimation of covariance with operator norm are not on par with other minimax methods. An interesting direction would be to investigate whether training \texttt{MIWAE} and \texttt{GAIN} with different metrics may improve the operator norm performance. 
        
        We refer to App. \ref{app:MAR} for more experiments in the MAR setting of \citep[Annex 3]{matteiMIWAEDeepGenerative2019} which led to similar conclusions. These results confirm that imputation of missing values is not mandatory for accurate estimation of the covariance operator. Another viable option is to apply a debiasing correction to the empirical covariance computed on the original data containing missing values. The advantage of this approach is its low computational cost even in high-dimension.

    \begin{table}[]
        \centering
        \caption{Execution time of the covariance estimation procedures (in milliseconds) with $n=300$ averaged over all values of the contamination rate $\delta$ and $20$ repetitions.}
        \begin{tabular}{ c | c c c }
            \toprule
            method & $p=50$ & $p=100$ & $p=500$ \\
            \midrule
            MV (ours) & $0.29 \pm 0.03$ & $0.49 \pm 0.08$ & $9.7\pm4.5$ \\
            KNNImputer (KNN) & $26 \pm 9.8$ & $45\pm 17$ & $470 \pm 190$\\
            IterativeImputer (II) & $940 \pm 350$ & $2,800\pm 900 $ & $3.7\times 10^5 \pm 1.1\times 10^5$
            \\
            Gain & $6,900 \pm 480$ & $1.1 \times 10^4 \pm 250$ & $8.8\times 10^4 \pm 1.1\times 10^3$\\
            MIWAE & $5.1\times 10^4 \pm 2.8\times 10^3$ & $6.7 \times 10^4 \pm 550$ & $1.77 \times 10^5\pm 5.8 \times 10^3$\\
            \bottomrule
        \end{tabular}
        \label{tab:syn_mv_exec}
    \end{table}

    \begin{figure}
        \begin{minipage}{0.45\textwidth}
            \centering
            \includegraphics[width=\textwidth]{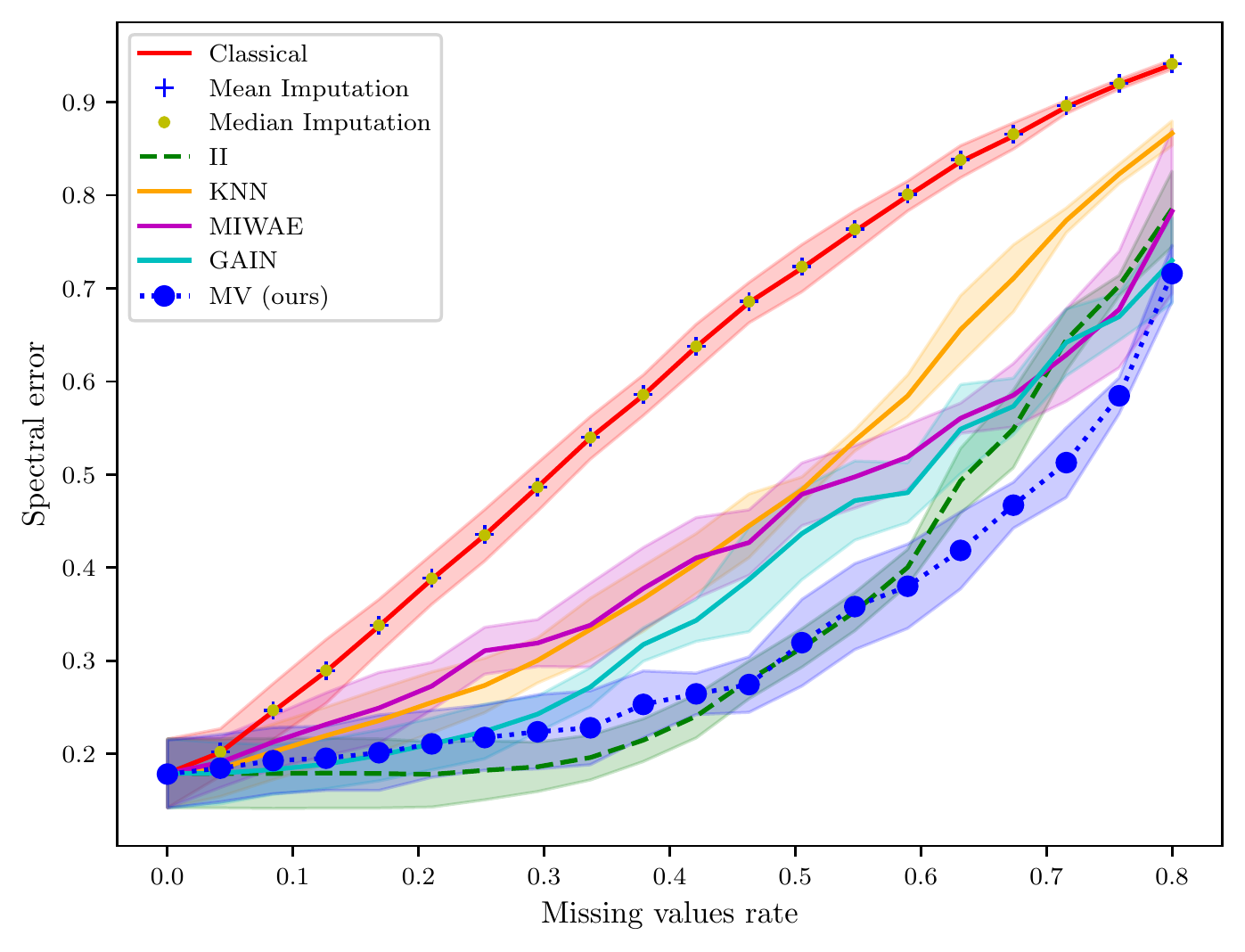}
            \caption{Estimation error on a synthetic dataset with $p=50$, $n=300$, $r\left(\Sigma\right)=5$.}
            \label{fig:syn_mv_error}
        \end{minipage}
        \hfill
        \begin{minipage}{0.45\textwidth}
            \centering
            \includegraphics[width=\textwidth]{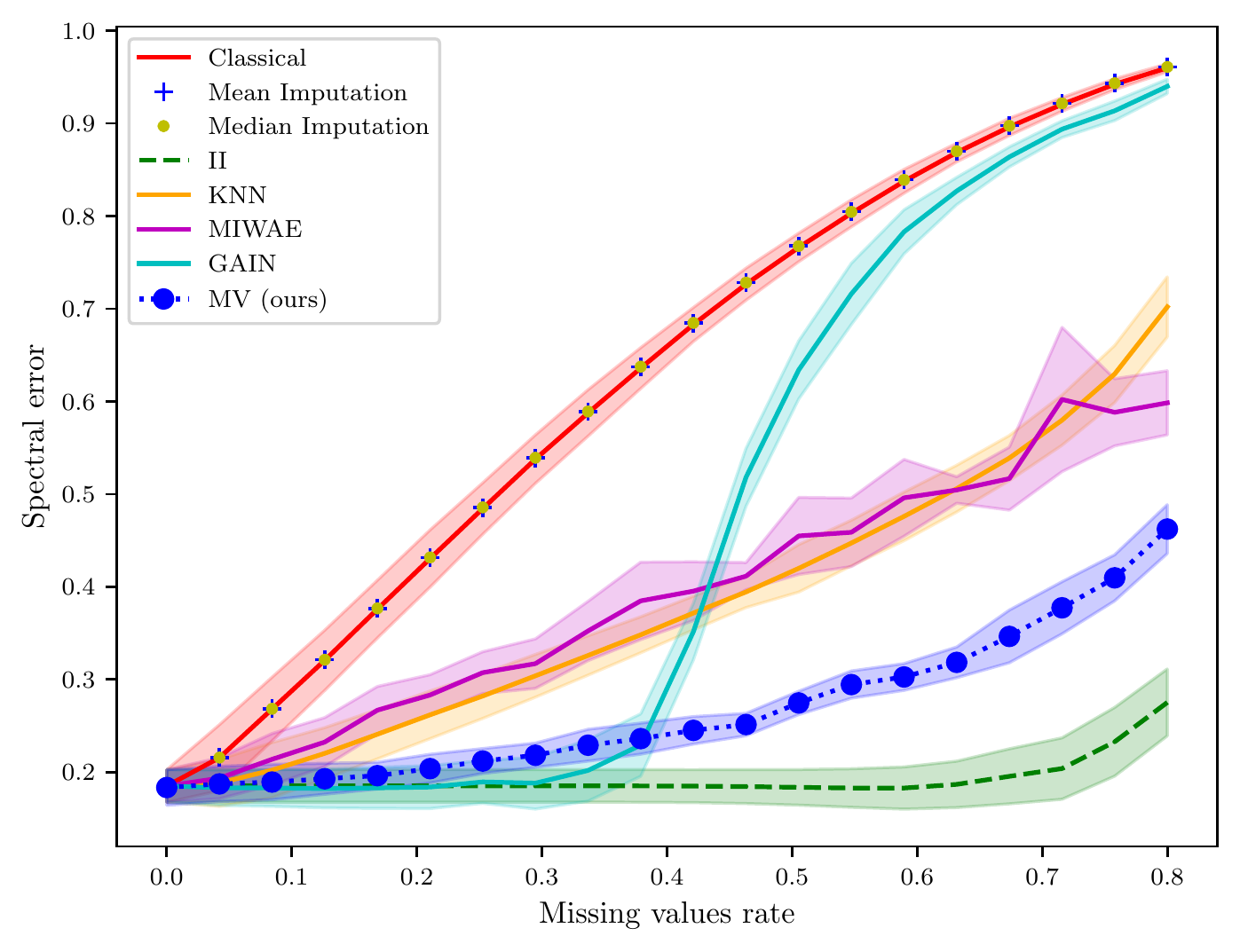}
            \caption{Estimation error on a synthetic dataset with $p=500$, $n=300$, $r\left(\Sigma\right)=5$.}
            \label{fig:gauss4estim-missing}
        \end{minipage}
        \end{figure}

    \subsection{Cell-wise contamination}

    \paragraph{Methods tested.} 

        Our baselines are the empirical covariance estimator applied without care for contamination and an oracle which knows the position of every outlier, deletes them and then computes the \texttt{MV} bias correction procedure \eqref{eq:debiased_Covest}. In view of Theorems \ref{th:upper} and \ref{th:lower}, this oracle procedure is the best possible in the setting of cell-wise contamination. Hence, we have a practical framework to assess the performance of any procedure designed to handle cell-wise contamination. 

The SOTA methods in the cell-wise contamination setting are the \texttt{DI} (Detection-Inputation) method \cite{raymaekersHandlingCellwiseOutliers2020} and the \texttt{TSGS} method (Two Step Generalised S-estimator) \cite{agostinelliRobustEstimationMultivariate2014}. Both these methods were designed to work in the standard setting $n>p$ but cannot handle the high-dimensional setting as we already mentioned. Nevertheless, we included comparisons of our methods to them in the standard setting $n>p$. The code for \texttt{DI} and \texttt{TSGS} are from the \texttt{R} packages \texttt{cellwise} and \texttt{GSE} respectively.

We combine the \texttt{DDC} detection procedure  \cite{rousseeuwDetectingDeviatingData2018} to first detect and remove outliers with several estimators developed to handle missing values.
Our main estimators are \texttt{DDCMV} (short for Detecting Deviating Cells Missing Values), which uses first \texttt{DDC} and then computes the debiaised covariance estimator \eqref{eq:debiased_Covest} on the filtered data, and \texttt{tailMV}, which detects outliers through thresholding and then uses again \eqref{eq:debiased_Covest}. But we also proposed to combine the \texttt{DDC} procedure with imputation methods \texttt{KNNI}, \texttt{II}, \texttt{GAIN} and \texttt{MIWAE} and finally compute the standard covariance estimator on the completed data. Hence we define four additional novel robust procedures which we call \texttt{DDCKNN}, \texttt{DDCII}, \texttt{DDCGAIN} and \texttt{DDCMIWAE}. To the best of our knowledge, neither the first approach combining filtering with debiasing nor the second alternative approach combining filtering with missing values imputation have never been tested to deal with cell-wise contamination.
A detailed description of each method is provided in App. \ref{sec:methods}.

    \paragraph{Outlier detection and estimation error under cell-wise contamination on synthetic data.}

        We showed that the error of a covariance estimator under cell-wise contamination depends on the proportion of remaining outliers after a filtration. In Table \ref{tab:conta_dirac} we investigate the filtering power of the \texttt{Tail Cut} and \texttt{DDC} methods in presence of Dirac contamination. We consider the cell-wise contamination setting \eqref{eqn:contaminated} in the most difficult case $\varepsilon=1$ which means that an entry is either correctly observed or replaced by an outlier (in other words, the dataset does not contain any missing value). For each values of $\delta$ in a grid, the quantities $\hat{\delta}$ and $\hat{\varepsilon}$ are the proportions of true entries and remaining contaminations after filtering averaged over $20$ repetitions. The DDC based methods are particularly efficient since the proportion of Dirac contamination drops from $1-\delta$ to virtually $0$ for any $\delta\geq 0.74$. In Figures \ref{fig:low_dim_estim} and \ref{fig:dirac4estim}, we see that the performance of our method is virtually the same as the oracle \texttt{OracleMV} as long as the filtering procedure correctly eliminates the Dirac contaminations. As soon as the filtering procedure fails, the statistical accuracy brutally collapses and our \texttt{DDC} based estimators no longer do better than the usual empirical covariance. In Table \ref{tab:conta_gauss} 
        and Figure \ref{fig:gauss4estim}, we repeated the same experiment but with a centered Gaussian contamination. Contrarily to the Dirac contamination scenario, we see in Figure \ref{fig:gauss4estim} that the statistical accuracy of our \texttt{DDC} based methods slowly degrades as the contamination rate increases but their performance remains significantly better than that of the usual empirical covariance.

      \begin{table}[t]
        \centering
        \caption{We consider contaminated data following model \eqref{eqn:contaminated} contaminated with a Dirac contamination of high intensity with $\varepsilon=1$ and for several values of $\delta$ in a grid. For each $\delta$, we average the proportion of real data $\hat{\delta}$ and contaminated data $\hat{\varepsilon}$ after filtering over $20$ repetitions. Values are displayed in percentages ($\hat{\delta}$ must be high, $\hat{\varepsilon}$ low)). STD stands for standard deviation.
        }
        \vskip 0.15in
        \begin{small}
        \begin{sc}
        \scalebox{0.8}{
        \begin{tabular}{c || c c | c c | c c | c c| c c | c c}
        \toprule
           Contamination   &  \multicolumn{4}{c|}{Tail cut} & \multicolumn{4}{c|}{DDC $99\%$} & \multicolumn{4}{c}{DDC $90\%$}\\
            rate ($1-\delta$) & $\hat{\delta}$ & std &  $\hat{\varepsilon}$ & std & $\hat{\delta}$ & std &  $\hat{\varepsilon}$ & std & $\hat{\delta}$ & std & $\hat{\varepsilon}$ & std\\
             \midrule
            0.1 $\%$ & 99.6 & 0.023 & 0.000 & 0.000 & 99.1 & 0.029 & 0.000 & 0.000 & 94.8 & 0.054 & 0.00 &  0.00 \\
            1$\%$ &  98.8 & 0.027 & 0.000 & 0.000 & 98.2 &  0.037 & 0.000 &  0.00 &  94.3 & 0.102 & 0.00 &  0.00 \\
            5$\%$ & 94.9 & 0.013 & 0.000 & 0.000 & 94.6 & 0.018 &  0.000 &  0.000 & 91.8 & 0.060 & 0.00 &  0.000 \\
            10$\%$ & 90.0 & 0.004 & 0.000 & 0.000 & 89.9 & 0.016 & 0.00 & 0.000 & 88.2 & 0.109 & 0.000 & 0.000 \\
            20$\%$ & 80.0 & 0.000 & 20.0 & 0.000 & 80.0 & 0.003 & 0.017 & 0.035 & 79.4 & 0.035 & 0.009 & 0.022 \\
            30$\%$ & 70.0 & 0.000 & 30.0 & 0.000 & 70.0 & 0.001 & 3.48 & 2.19 & 69.9 & 0.015 & 2.930 & 2.31 \\
             \bottomrule
         \end{tabular}}
        \end{sc}
        \end{small}
        \vskip -0.1in
        \label{tab:conta_dirac}
    \end{table}

        \begin{figure}
        \begin{minipage}{0.45\textwidth}
            \centering
            \includegraphics[width=\textwidth]{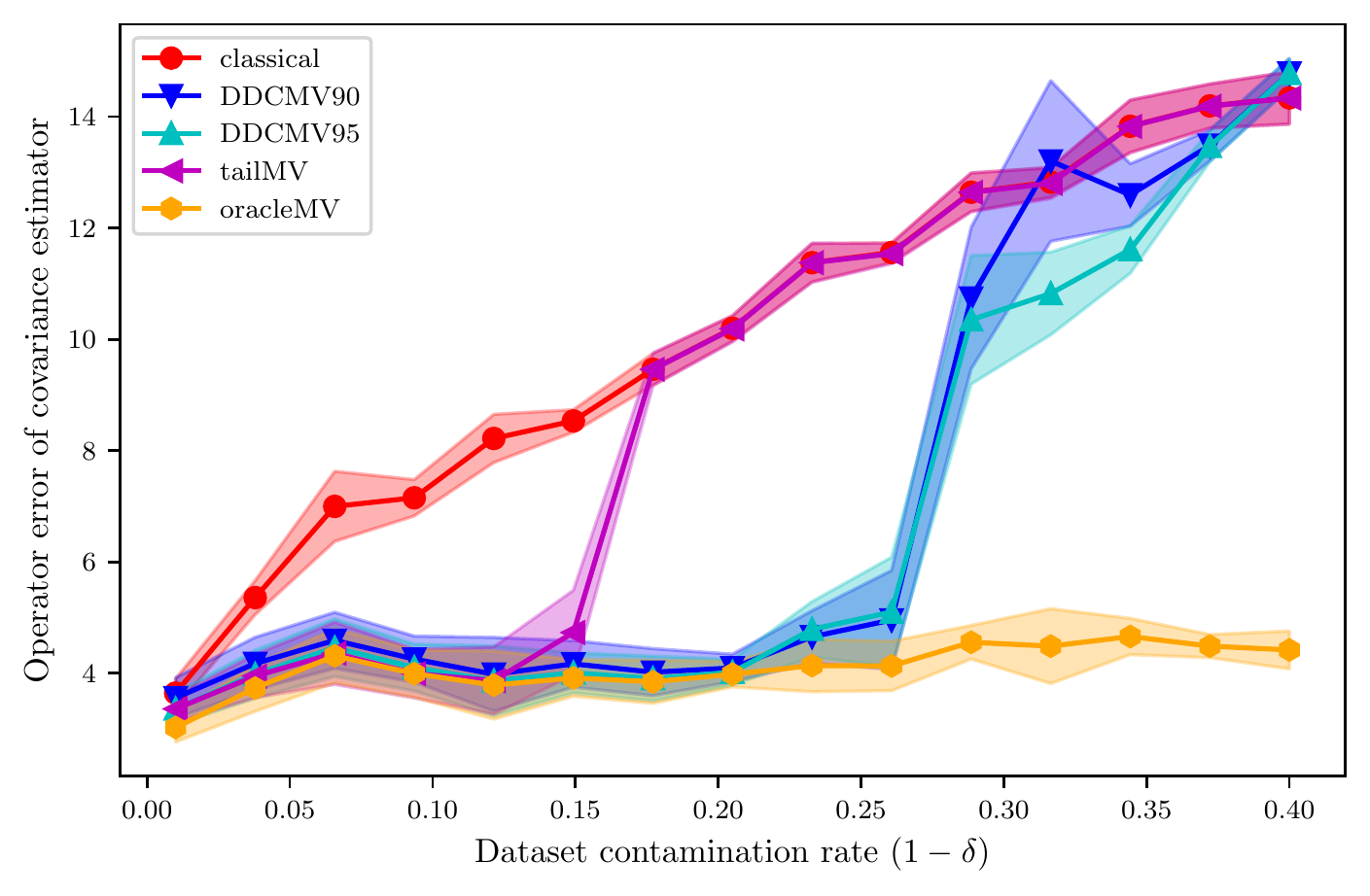}
            \caption{Estimation error as a function of the contamination rate for $n=500$, $p=400$, $\mathbf{r}(\Sigma)=5$ and Dirac contamination 
            .}
            \label{fig:dirac4estim}
        \end{minipage}
        \hfill
        \begin{minipage}{0.45\textwidth}
            \centering
            \includegraphics[width=\textwidth]{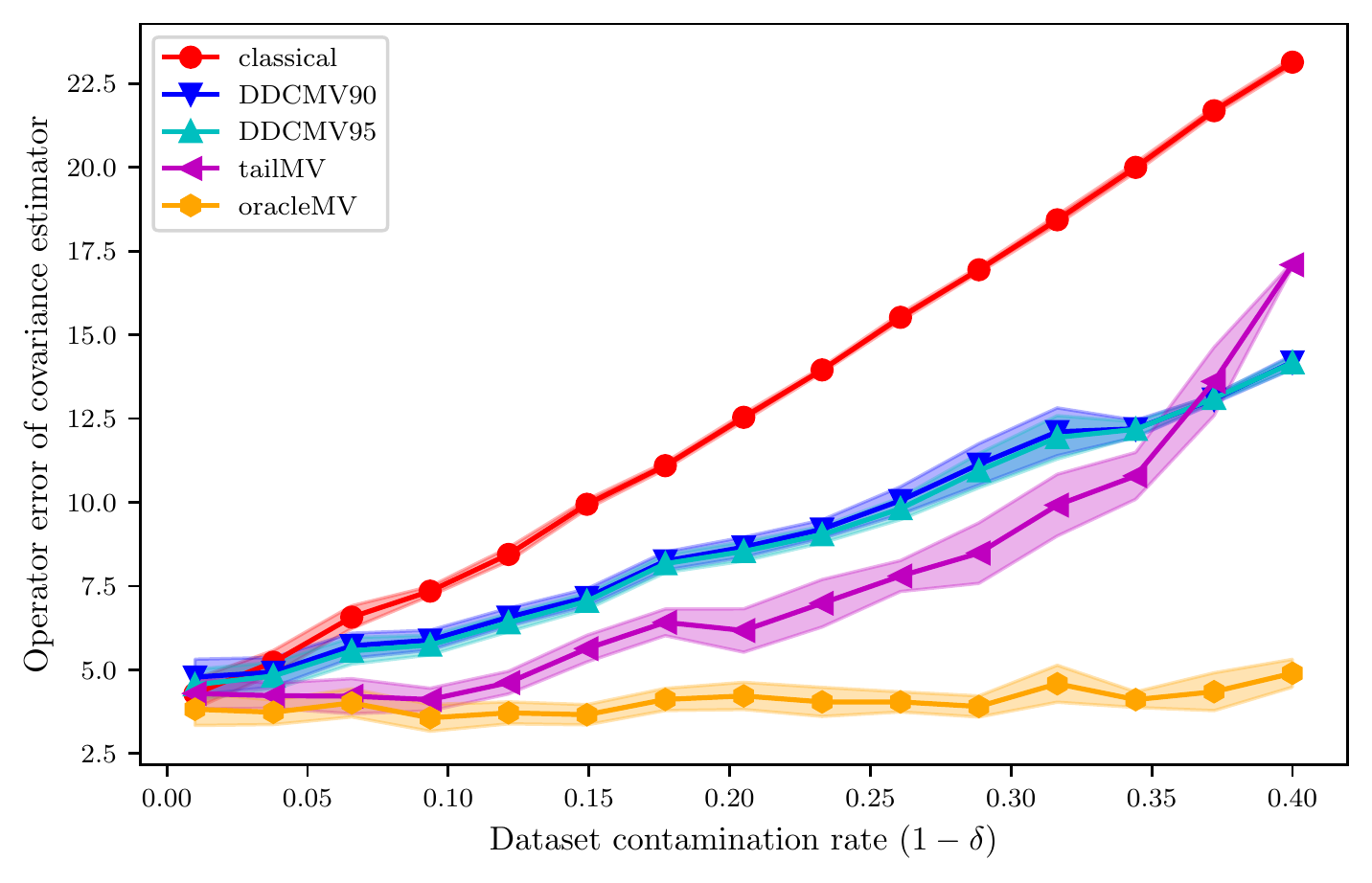}
            \caption{Estimation error as a function of the contamination rate for $n=500$, $p=400$, $\mathbf{r}(\Sigma)=5$ and Gaussian contamination 
            .}
            \label{fig:gauss4estim}
        \end{minipage}
        \end{figure}

    \subsection{The effect of cell-wise contamination on real-life datasets}

        We tested the methods on $8$ datasets from sklearn and Woolridge's book on econometrics \cite{wooldridgeIntroductoryEconometricsModern2016}.  These are low dimensional datasets (less than $20$ features) representing various medical, social and economic phenomena. We also included $2$ high-dimensional datasets. See App. \ref{sec:real_data} for the list of the datasets.

        \begin{figure}
            \begin{minipage}[t]{0.45\textwidth}
                \centering
                \includegraphics[width=\textwidth]{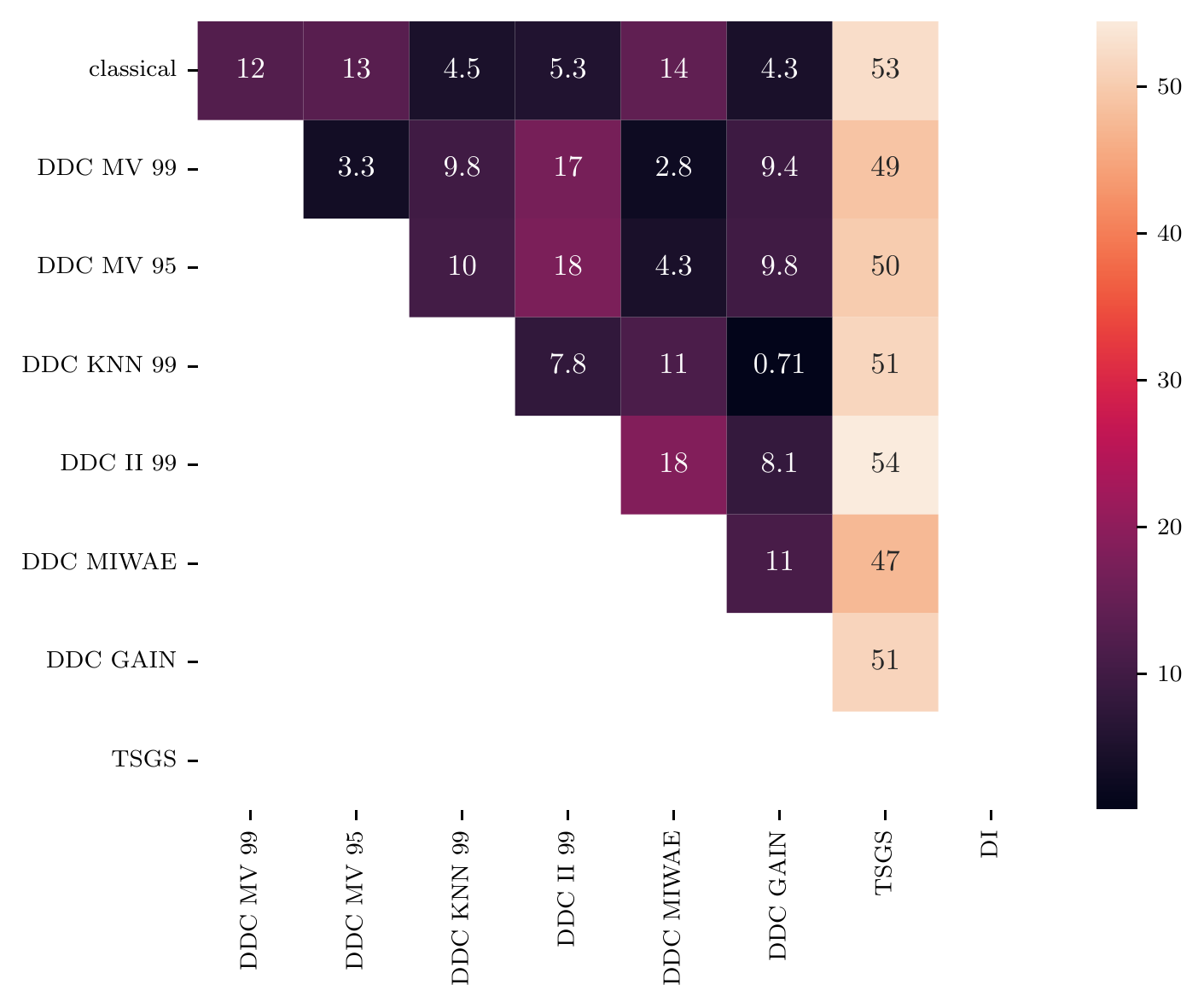}
    \caption{DI fails on ATTEND since the covariance matrix is approximately low rank. The dataset has only $8$ features and the effective rank of its covariance matrix is below $2$.}
    \label{fig:attend}
            \end{minipage}
            \hfill
            \begin{minipage}[t]{0.45\textwidth}
                \centering
            \includegraphics[width=\textwidth]{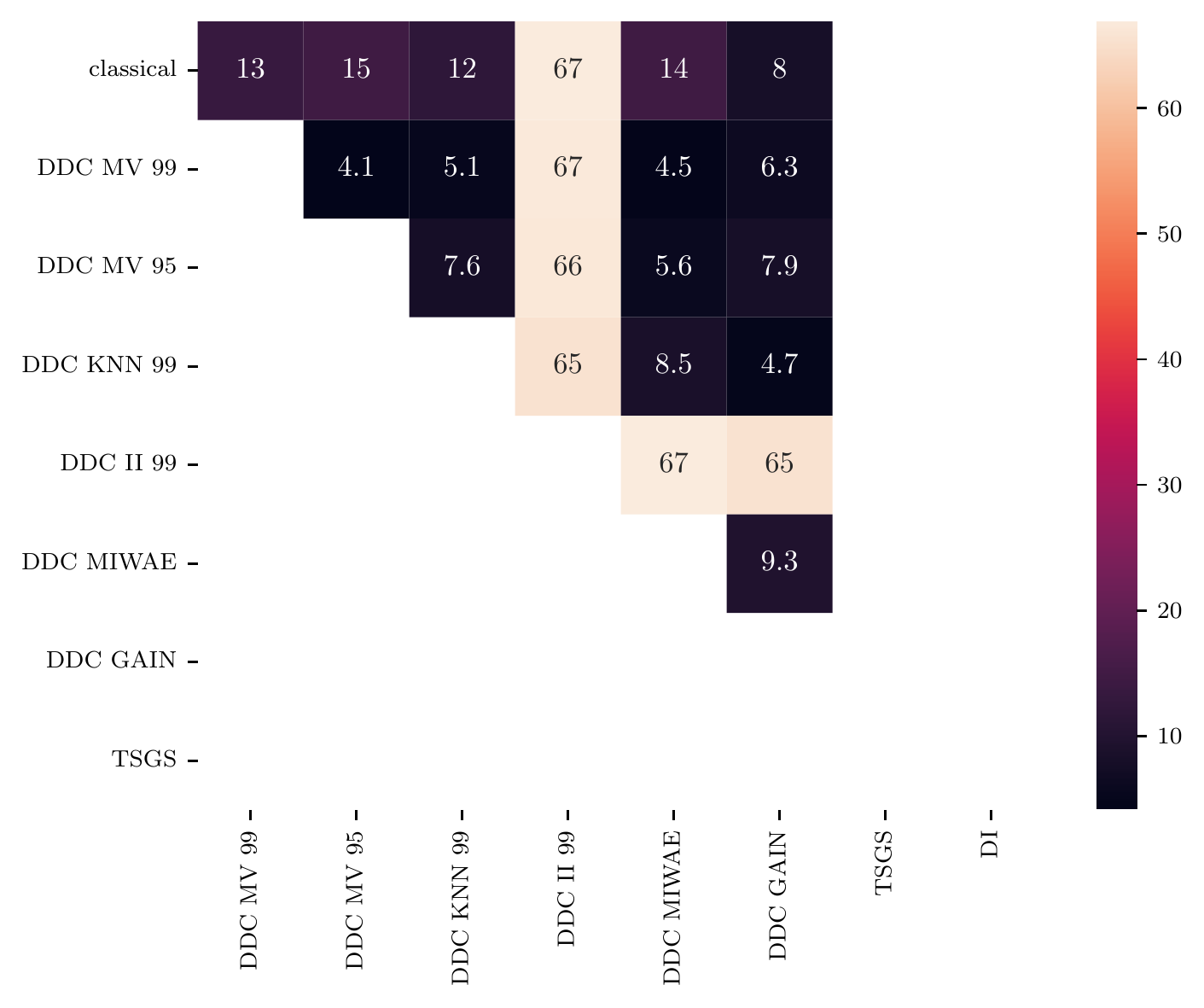}
            \caption{Woolridge's CEOSAL dataset fails both TSGS and DI with its dimension of $13$ and effective rank of around $2.5$.}
            \label{fig:ceosal}
            \end{minipage}
        \end{figure}

One interesting observation is that the instability of Mahalanobis distance-based algorithms is not limited to high-dimensional datasets. Even datasets with a relatively small number of features can exhibit instability. This can be seen in the performance of \texttt{DI} on the Attend dataset, as depicted in Figure \ref{fig:attend}, where it fails to provide accurate results. Similarly, both \texttt{TSGS} and \texttt{DI} fail to perform well on the CEOSAL2 dataset, as shown in Figure \ref{fig:ceosal}, despite both datasets having fewer than $15$ features. 

       On the Abalone dataset, once we have removed 4 obvious outliers (which are detected by both DDC and the tail procedure), all estimators reached a consensus with the non-robust classical estimator, meaning that this dataset provides a ground truth against which we can evaluate and compare the performance of robust procedures in our study. To this end, we artificially contaminate $5\%$ of the cells at random in the dataset with a Dirac contamination and compare the spectral error of the different robust estimators. As expected, \texttt{TSGS} and all our new procedures succeed at correcting the error, however \texttt{DI} becomes unstable (see Table \ref{tab:abalone}). \texttt{DDC MIWAE} is close to SOTA \texttt{TSGS} for cellwise contamination and \texttt{DDC II} performs better. We also performed experiments on two high-dimensional datasets, where our methods return stable estimates of the covariance (\texttt{DDCMV99} and \texttt{DDCMV95} are within $\approx 3\%$ of each other) and farther away from the classical estimator (See Figures \ref{fig:sp500} and \ref{fig:nasdaq}
       ). Note also that \texttt{DDCII}'s computation time explodes and even returns out-of-memory errors due to the high computation cost of \texttt{II} that we already highlighted in Table \ref{tab:syn_mv_exec}. 


\begin{figure}
    \begin{minipage}[t]{0.45\textwidth}
        \centering
    \includegraphics[width=\textwidth]{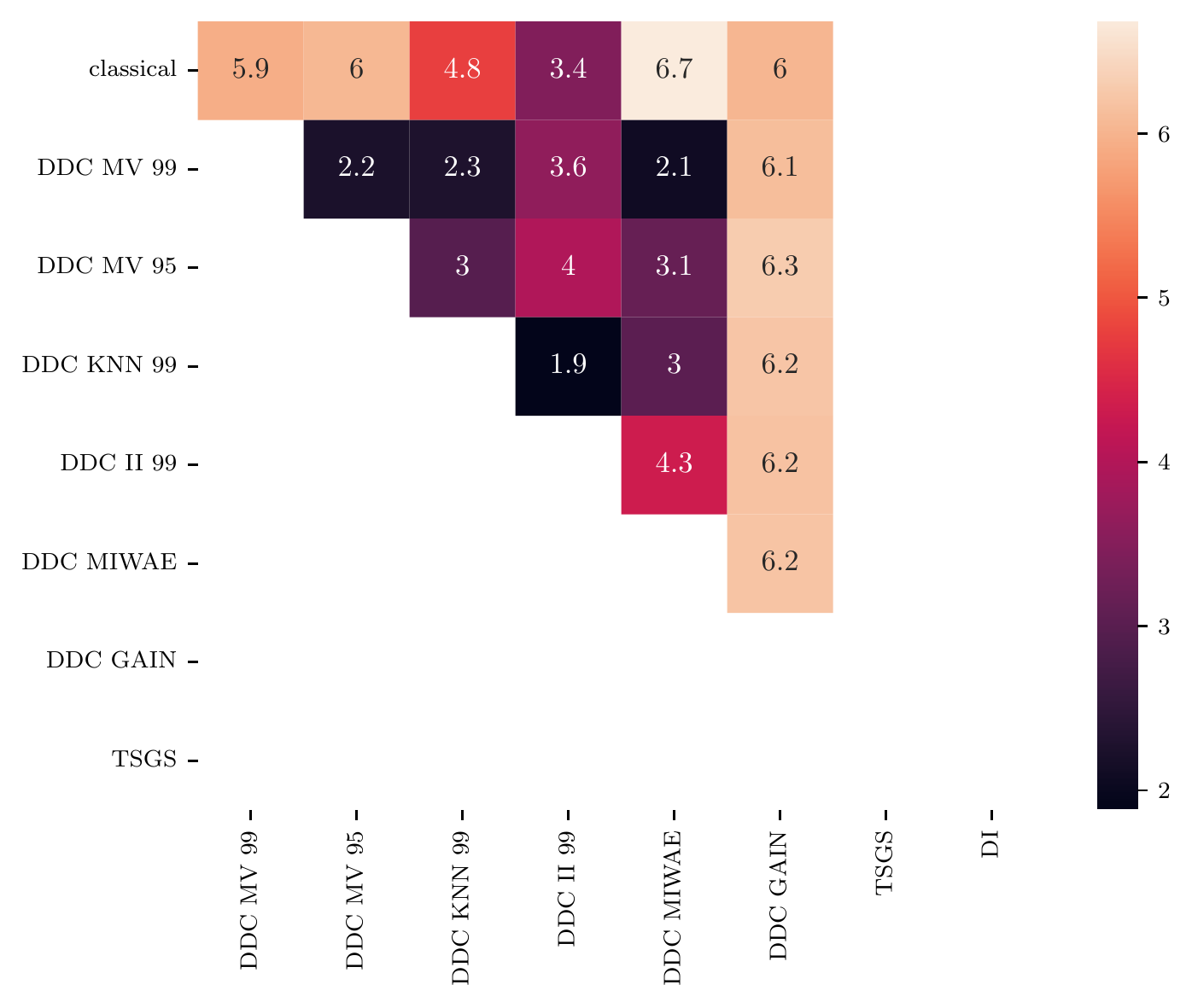}
    \caption{Relative spectral difference (in $\%$) between covariance estimators on SP500 stock returns over 2021 and 2022. On high-dimensional data, \texttt{DDCII} becomes inconsistent with the other procedures, maybe because \texttt{II} does not scale well with dimension.}
    \label{fig:sp500}
    \end{minipage}
    \hfill
    \begin{minipage}[t]{0.45\textwidth}
    \centering
     \includegraphics[width=\textwidth]{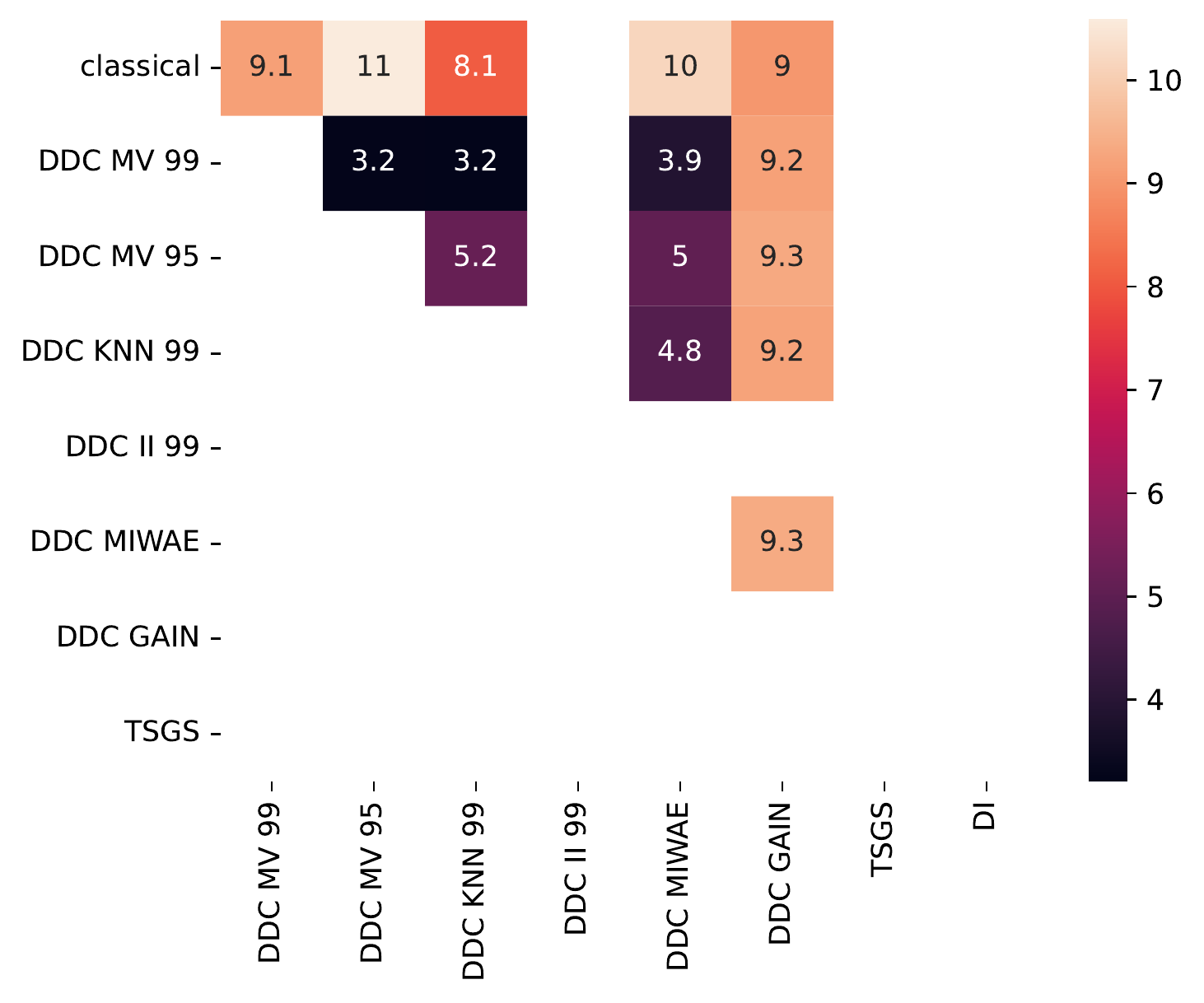}
    \caption{Relative spectral difference (in $\%$) between between covariance estimators on NASDAQ stock returns over 2021 and 2022. Here, \texttt{DDCII} fails due to out-of-memory errors.}
    \label{fig:nasdaq}
    \end{minipage}
\end{figure}

\begin{table}[ht]
    \centering
    \caption{Relative spectral difference (in $\%$) between estimated covariance matrices on Abalone with $5\%$ synthetic contamination ($\delta=0.95$, $\varepsilon=1$). On the cleaned dataset, all the robust estimators are very close to the empirical covariance (relative differences $< 5\%$), so we consider the empirical covariance matrix as the truth. Here the \texttt{DI} procedure fails probably due to numerical errors.}
    \scalebox{0.8}{
    \begin{tabular}{c|c c c c c c c c c}
         \toprule
         relative & Classical & \multirow{2}{*}{DDCMV99} & \multirow{2}{*}{DDCMV95} & \multirow{2}{*}{DDC II} & \multirow{2}{*}{DDC KNN} & DDC & DDC & \multirow{2}{*}{TSGS} & \multirow{2}{*}{DI}\\
         error to & estimator & & & & & MIWAE & GAIN\\  
         \midrule
         Truth &  12.8 & 4.12 &	6.81 &	\textbf{1.70} &	2.06 &	3.46 &	5.06 &	3.06 & 8.85\\
         \textit{std} & \textit{0.45} &	\textit{0.29} &	\textit{0.26} &	\textbf{\textit{0.10}} &	\textit{0.092} &	\textit{0.18} &	\textit{0.35} &	\textit{0.21} & 	\textit{1.48}\\
         \midrule
         Classical & - & 13.1 &	14.3 &	13.0 &	13.0 &	12.9 &	13.1 &	13.4 &	14.9\\
         DDCMV99 & - & - & 2.99 & 2.52 & 2.22 &	1.87 &	2.66 &	5.44 & 	8.79\\
         DDCMV95 & - & - & - & 5.27 & 5.03 & 4.04 & 3.71 &	8.28 & 9.99 \\
         DDC II & - & - & - & - & 0.465 & 1.88 & 3.49 &	3.27 & 	8.28\\
         DDC KNN & - & - & - & - & - & 1.58 & 3.19 & 3.46 &	8.15\\
         DDC MIWAE & - & - & - & - & - & - & 1.70 & 4.50 & 7.22\\
         DDC GAIN & - & - & - & - & - & - &  - & 5.97 &	6.71\\
         TSGS & - & - & - &  - & - & - & - & - & 6.94 \\
         \bottomrule
    \end{tabular}}
    \label{tab:abalone}
\end{table}




\section{Conclusion and future work}

    In this paper, we have extended theoretical guarantees on the spectral error of our covariance estimators robust to missing data to the missing at random setting. We have also derived the first theoretical guarantees in the cell-wise contamination setting. We highlighted in our numerical experimental study that in the missing value setting, our debiased estimator designed to tackle missing values without imputation offers statistical accuracy similar to the SOTA \texttt{IterativeImputer} for a dramatic computational gain. We also found that SOTA algorithms in the cell-wise contamination setting often fail in the standard setting $p<n$ for dataset with fast decreasing eigenvalues (resulting in approximately low rank covariance), a setting which is commonly encountered in many real life applications. This is due to the fact that these methods use matrix inversion which is unstable to small eigenvalues in the covariance structure and can even fail to return any estimate. In contrast, we showed that our strategy combining filtering with estimation procedures designed to tackle missing values produce far more stable and reliable results. In future work, we plan to improve our theoretical upper and lower bounds in the cell-wise contamination setting to fully clarify the impact of this type of contamination in covariance estimation.

\paragraph{Acknowledgements.} This paper is based upon work partially  supported by the Chaire {\em Business Analytic for Future Banking} and EU Project ELIAS under grant agreement No. 101120237.

\bibliographystyle{plain}
\bibliography{main.bib}

\appendix
\onecolumn

Appendix \ref{sec:data_gen} presents the synthetic data generation procedure used throughout our experiments. Appendix \ref{sec:real_data} and in particular Table \ref{tab:datasets} list the real life datasets presented in the paper. The cell-wise contamination correction methods are shown in Appendix \ref{sec:methods}, with the DDC algorithm of \cite{rousseeuwDetectingDeviatingData2018} further detailed in Appendix \ref{sec:DDC} for convenience. The upper bound proofs can be found in Appendix \ref{app:upper_bound} and the lower bound proofs in Appendix \ref{app:lower_bound}, so that similar proof techniques can be grouped together for clarity. Additional technical elements of these proofs are collected in Appendix \ref{sec:other_proofs} when we felt that they impacted the latter's readability. Finally, we show the full results of our experiments in Appendix \ref{app:tables}.

\begin{table}[]
    \centering
    \caption{Notations}
    \scalebox{0.74}{
    \begin{tabular}{c|c || c | c}
        \toprule
        Symbol & Description & Symbol & Description \\
        \midrule
        $X$ & The random variable of interest & $\widehat{\Sigma}$ & Unbiased estimator of the covariance of $X$\\
        $Y$ & The observed contaminated random variable & $\widehat{\Sigma}^Y$ & Empirical covariance of $Y$\\
        $p$ & Dimension of the random variable & $\Lambda$ & Noise covariance matrix\\
        $n$ & Number of samples & $\norm{X}_2$ & Vector $L^2$ norm (Euclidean norm)\\
        $\delta$ & Probability that a cell be observed correctly & $\norm{\Sigma}$ & Operator norm of $\Sigma$\\
        $d$ & Bernoulli random variable of probability $\delta$ & $\norm{\Sigma}_F$ & Frobenius norm of $\Sigma$\\
        $\varepsilon$ & Probability that an unobserved cell be contaminated & $\psinorm{\Sigma}{\alpha}$ & $\alpha$-Orlicz norm of $\Sigma$\\
        $e$ & Bernoulli random variable of probability $\varepsilon$ & $\odot$ & Hadamard or term by term product of matrices\\
        $\Sigma$ & True covariance matrix of $X$ & $\otimes$ & Outer product of vectors\\
        $\Sigma^Y$ & True covariance matrix of $Y$ & $\mathbb{I}$ & Indicator function\\
        $\bm{r}(\Sigma)$ & Effective rank of $\Sigma$ & $\lesssim$ & Domination with regard to an absolute constant\\
        \bottomrule
    \end{tabular}}
    
    \label{tab:my_label}
\end{table}

\section{Synthetic data generation}
\label{sec:data_gen}

    We generate synthetic datasets of $n$ realisations of a multivariate centered normal distribution. Its covariance matrix is defined as follows. We first set the eigenvalues as $ \lambda_j = \exp \left( -j / r \right)$ for $j\in\lbrace 1, p\rbrace$, where $r$ is the requested effective rank of the matrix. This approximation guaranties that the true effective rank is below $r+1$ for $r << p$. Then, using the \texttt{ortho-group} tool from \texttt{scipy.stats}, we create a random orthonormal matrix $H$ and set $\Sigma = H \text{diag}(\lambda)H^\top$, which is symmetric and of low effective rank at most $r+1$. Finally, we divide $\Sigma$ by its largest diagonal term so that the variances of the marginals be closer to $1$.

    We contaminate our synthetic datasets using a binary mask obtained by computing the realisation of $n\times p$ i.i.d. bernoulli random variables. We fill the resulting missing data with either $n$ samples of a isotropic gaussian of covariance $\sigma I_p$, where $\sigma$ is the strength of the contamination (which we call the Gaussian contamination) or a $n\times p$ array of value $\pm\sigma$ (which we call the Dirac contamination). Let $\xi$ be a random vector following one of those two contaminations, the data we feed all algorithms is then $Y = \text{mask} \odot X + (1-\text{mask}) \odot \xi$.

\section{Real life data set}
\label{sec:real_data}

For our real data experiments, we removed any categorical variable from the datasets since this work focuses on covariance estimation.
We also applied a log transform to skewed variables to ensure that they are sub-Gaussian. The list of datasets can be found in Table \ref{tab:datasets}. Finally, the Abalone dataset contains four obvious outliers that we removed in our experiments (although they were easily detected by both DDC and the thresholding procedure) in order to obtain a perfect dataset (no missing values, no contaminations) allowing us to compute the ground truth covariance. We have then injected missing values and cell-wise contaminations in this dataset and compared our robust procedures to the ground truth. We also note that the three UCI datasets were downloaded from sklearn.

\begin{table}[]
    \centering
    \caption{Datasets used in our real-life experiments. $p$ and $n$ are indicated after dropping categorical features and obvious sample-wise outliers.}
    \scalebox{0.9}{
    \begin{tabular}{c|c| c c c c}
        \toprule
        Name & Source & $p$ & $n$ & $r\left(\Sigma\right)$ & Description \\
        \midrule
        Abalone & UCI & 7 & 4173 & 1.0 & Caracteristics of abalone specimens\\
        Breast Cancer & UCI & 13 & 178 & 2.3 & Data on cell nuclei \\
        Wine & UCI & 30 & 69 & 2.8 & Chemical data on wine varieties \\
        Cameras & R & 11 & 1038 & 2.7 & Camera caracteristics over different models\\
        Attend & \cite{wooldridgeIntroductoryEconometricsModern2016} & 8 & 680 & 2.0 & Class attendance\\
        Barium & \cite{wooldridgeIntroductoryEconometricsModern2016} & 11 & 131 & 2.4 & Barium exports\\
        CEOSAL2 & \cite{wooldridgeIntroductoryEconometricsModern2016} & 13 & 177 & 2.5 & Firm accountancy data\\
        INTDEF & \cite{wooldridgeIntroductoryEconometricsModern2016} & 12 & 49 & 2.2 & USA deficit\\
        SP 500 & yfinance & 496 & 502 & 2.7 & Returns of SP 500 companies in 2021/2022\\
        NASDAQ & yfinance & 1442 & 502 & 4.0 & Returns of NASDAQ companies in 2021/2022\\
        \bottomrule
    \end{tabular}}
    \label{tab:datasets}
\end{table}
            
\section{Methods compared in the cell-wise contamination setting}
\label{sec:methods}

\subsection{Baseline methods}
 \paragraph{Classical} denotes the empirical covariance estimator applied without care for contamination. We expect all other methods to perform better than it.

        \paragraph{oracleMV}
        is an oracle that knows which cells are contaminated. This method shows the performance of our corrected estimator in the case of a perfect outlier detection algorithm, hence providing an idea of the optimal precision attainable with regard to the available information.

\subsection{Our methods}
        
        \paragraph{tailMV} or tail Missing Values, is an estimator built by deleting extreme values in the dataset. It is actually one of the intermediary steps of DDC and we wanted to test how efficient it was on its own. We use the robust Huber estimator of the python package \texttt{Statsmodel.robust} \cite{huberRobustStatistics2nd2009} to compute the standard deviation of each marginal and eliminate any cell with value above $3$ times these estimates.

        \paragraph{DDCMV} short for Detecting Deviating Cells Missing Values, is an estimator built using the DDC detection procedure of \cite{raymaekersHandlingCellwiseOutliers2020}, where detected outliers are removed and considered as missing values. A detailed description of DDC is provided in appendix \ref{sec:DDC}. We then apply our corrected covariance estimator. We will add to the name of the method the quantile at which we consider a data as an outlier (DDCMV99 uses the 99-percentile of $\chi^2_1$ for instance). When nothing is mentioned, assume that DDCMV99 is used. In our experiments, we use the \texttt{R} implementation found in the package \texttt{cellWise}, whose results are then sent to a python script for formatting.
        \paragraph{DDCKNN} detects outliers with the DDC procedure, removes them and imputes the missing values using the k-nearest neighbour procedure of \cite{troyanskayaMissingValueEstimation2001} as implemented in sklearn under the name \texttt{KNNImputer}.
        \paragraph{DDCII} also detects and removes outliers with the DDC procedure, then imputes the missing values using sklearn's \texttt{Iterative Imputer} class.

        \paragraph{DDCGAIN} is the combination of the DDC algorithm for outlier detection followed by the GAIN deep imputation method of \cite{yoonGAINMissingData2018}.

        \paragraph{DDCMIWAE} is the combination of the DDC algorithm for outlier detection followed by the MIWAE deep imputation method of \cite{matteiMIWAEDeepGenerative2019}.

\subsection{SOTA methods for cell-wise contamination}
        \paragraph{DI} or Detection Imputation \cite{raymaekersFastRobustCorrelation2021} Is an iterative algorithm made of two alternating steps inspired by the Expectation Maximisation (EM) algorithm. The first detects outliers with regard to a previously estimated covariance matrix, then the second computes a new covariance matrix having removed the previously detected outliers using the M step of EM, but with bias correction. This new matrix is then the basis for the next detection step and so on. The authors found their algorithm to have a $O(Tnp^3)$ complexity, with $T$ the number of iterations, and make the assumption that the covariance matrix is of full rank to perform matrix inversion, both facts that make it difficult to use in high dimensions. 

        \paragraph{TSGS} or Two Steps Generalised S-estimator \cite{agostinelliRobustEstimationMultivariate2014} and \cite{leungRobustRegressionEstimation2016} is also based on a two step process of detection then correction. Detection is based on the same DDC procedure while the estimation phase is based on the Generalised S-estimator of \cite{danilovRobustEstimationMultivariate2012}. S-estimators are based on the Mahalonobis distance and thus require the true covariance matrix to be invertible. This may lead to numerically instability in our approximately low rank setting. However, if the matrix is of full rank, the generalised version of these estimators are proven to be consistent in the Missing Completely At Random setting.

\section{The Detecting Deviating Cells algorithm}
\label{sec:DDC}
    
    This section is entirely based on \cite{rousseeuwDetectingDeviatingData2018}, whose algorithm we describe here for convenience. DDC (Detecting Deviating Cells) is a 7 steps algorithm. In the following, let $(X_i^j)_{i \in [n], j \in [p]}$ be our dataset of $n$ samples from data with dimension $p$.
    \paragraph{Step 1: standardisation} We start by assuming that the $X_i$ follow a normal distribution and we set \[Z_i^j = \frac{X_i^j - \mu_X^j}{\sigma_X^j}\] with $\mu_X^j$ being the empirical mean of marginal $j$, and $\sigma_X^j$ its standard deviation.
    \paragraph{Step 2: cutoff} DDC sets to NA all values of $Z_i^j$ if \[\vert Z_i^j \vert \geq \sqrt{\chi^2_{1, p}}\]
    with $\chi^2_{1,p}$ the $p^{\text{th}}$ centile of a $\chi^2_1$ distribution, where $p = 99\%$ by default.
    \paragraph{Step 3: bivariate relationship} The algorithm then computes the correlation between each couple of marginals. If $\vert \rho_i(Z^j, Z^k) \vert \leq 0.5$; set $b_{jk} = 0$. Otherwise, 
    \[ b_{jk} = \text{slope}(Z^j \vert Z^k)\]
    with $\text{slope} (x\vert y)$ the robust slope in the linear regression of $x$ using $y$.
    \paragraph{Step 4: comparison} Then DDC tries to predict the expected values of each $Z_i^j$ according to a weighted mean of the values of the other marginals, using the previously computed correlations as weights.
    \[ \hat{Z}_i^j = G\left( \lbrace b_{jk} Z_i^k, k\in [p], k\neq j \rbrace \right) \]
    with $G$ the weighted mean using $\rho(Z^j, Z^h)$ as weights.
    \paragraph{Step 5: deshrinkage} DDC adjusts the mean to account for shrinkage.
    \begin{equation*}
        \begin{split}
        a_j & = \text{slope}(Z_i^j \vert \hat{Z}_i^j)\\
        Z^{j\star}_i & = a_j \hat{Z}_i^j
        \end{split}
    \end{equation*}
    \paragraph{Step 6: residual computation} Then, one can take the residuals:
    \[ r_i^j = \frac{Z_i^j - \hat{Z}_i^j}{\mu_{Z^j - \hat{Z}^j}}\]
    \paragraph{Step 7: destandardisation} Finally, DDC returns the data to its actual location and scale. The residuals can then be tested using a $\chi^2_1$ law to determine whether or not they are outliers.

\section{Missing at Random experiment}
\label{app:MAR}

To assess our estimator \eqref{eq:corr_hetero} in the heterogeneous missingness setting, we replicated the MAR experiment of \citep[Annex 3]{matteiMIWAEDeepGenerative2019}. In this experiment, the data is missing with a different probability for each feature. These probabilities are fixed prior to the experiment and depend on the data, although the bernoulli random variables are still independent to the data. Just as in \cite{matteiMIWAEDeepGenerative2019}, the probability $\delta_j$ that an $X^{j}$ is observed depends on the first 15 samples and:
\begin{equation}
    \delta_j = 1-\text{sigmoid}\left(\frac{1}{15}\sum_{j=1}^{15}x_j \right)
\end{equation}

We compared our debiasing estimator \eqref{eq:corr_hetero} to the traditional \texttt{KNNimputer}, \texttt{IterativeImputer} and the recent imputation methods \texttt{GAIN} and \texttt{MIWAE} which are expected to perform better in this setting. On the \texttt{Abalone} dataset (Table \ref{tab:MARabalone}), \texttt{MV} is the most accurate for the operator norm and \texttt{MIWAE} and \texttt{GAIN} are far behind and performs worse than \texttt{IterativeImputer} or \texttt{KNNimputer}. The \texttt{Breast Cancer} data was used both in \cite{matteiMIWAEDeepGenerative2019,yoonGAINMissingData2018}. We used the colab code provided by \cite{matteiMIWAEDeepGenerative2019} to implement \texttt{MIWAE} method. For \texttt{GAIN} we use the defaults parameters as in the HyperImpute library. We see in Table \ref{tab:MARBC} that \texttt{GAIN} is the second best method behind \texttt{MV} and is better than \texttt{IterativeImputer}. We also note that, in all our experiments, the computation times were far longer for \texttt{GAIN} and \texttt{MIWAE} than for our debiasing scheme \texttt{MV}.

\begin{table}[]
\begin{minipage}{0.5\textwidth}
    \centering
    \caption{Abalone}
    \begin{tabular}{c|c c}
    \toprule
        Method & mean error & std \\
        \midrule
        classical &  $59.33$ & $0.58$\\
        MV & $\bm{1.87}$ & $0.57$ \\
        II & $4.28$ & $1.44$\\
        KNN & $4.55$ & $1.31$ \\
        GAIN & $16.5$ & $2.09$\\
        MIWAE & $8.60$ & $0.62$ \\
        \bottomrule
    \end{tabular}
    
    \label{tab:MARabalone}
\end{minipage}
\begin{minipage}{0.5\textwidth}
    \centering
    \caption{Breast cancer}
    \begin{tabular}{c|c c}
        \toprule
        Method & mean error & std \\
        \midrule
        classical &  $88.63$ & $0.46$\\
        MV & $\bm{21.74}$ & $3.00$ \\
        II & $69.79$ & $2.53$\\
        KNN & $61.97$ & $10.04$ \\
        GAIN & $37.46$ & $5.78$\\
        MIWAE & $87.96$ & $0.43$ \\
        \bottomrule
    \end{tabular}
    
    \label{tab:MARBC}
\end{minipage}
\end{table}

\section{Proofs of upper bounds}
\label{app:upper_bound}
    \subsection{Tools and definitions}

    \subsubsection{Basic properties of random vectors}
    
        We recall the definition and some basic properties of sub-exponential random vectors.
    
        \begin{definition}
        \label{def:psinorm}
            For any $\alpha \geq 1$, the $\psi_\alpha$-norms of a real-valued zero mean random variable $V$ are defined as:
            $$ \psinorm{V}{\alpha} = \inf \lbrace u> 0, \expectation\exp\left(\vert V \vert^\alpha / u^\alpha \right) \leq 2 \rbrace $$
            We say that a random variable $V$ with values in \real is sub-exponential if $\psinorm{V}{\alpha} < \infty$ for some $\alpha \geq 1$. If $\alpha = 2$, we say that $V$ is sub-Gaussian.
        \end{definition}

    \begin{lemma}[Lemma 5.14 in \cite{vershyninIntroductionNonasymptoticAnalysis2011}]
    If a real-valued random variable $V$ is sub-Gaussian, then $V^2$ is sub-exponential. Indeed, we have: \[ \psinorm{V}{2}^2\leq \psinorm{V^2}{1} \leq 2 \psinorm{V}{2}^2\].
    \end{lemma}

        \begin{definition}
            The $\psi_\alpha$-norms of a random vector $X$ are defined as:
            \[\psinorm{X}{\alpha} = \sup_{x\in\real^p, \vert x \vert_2 = 1} \psinorm{\scalarproduct{X}{x}}{\alpha}, \qquad \alpha \geq 1\]
        \end{definition}

We will use the following definition of sub-Gaussian vectors that can be found in \cite{koltchinskiiConcentrationInequalitiesMoment2017}.
 \begin{definition}
    \label{def:subgaussian}
        A random vector $X \in \real^p$ is sub-Gaussian if and only if $\forall x \in \real^p$, $\psinorm{\scalarproduct{X}{x}}{2} \lesssim \norm{\scalarproduct{X}{x}}_{L^2}$.
    \end{definition}

We recall a version of Bernstein's inequality (see corollary 5.17 in \cite{vershyninIntroductionNonasymptoticAnalysis2011}):
        \begin{proposition}
            \label{th:bernstein}
            Let $Z_1, \dots Z_n$ be independent sub-exponential zero mean real-valued random variables. Set $K = \max_i \psinorm{Z_i}{1}$. Then, for $t>0$, with probability at least $1-e^{-t}$:
            \begin{equation}
                \left\vert n^{-1}\sum_{i=1}^n Z_i \right\vert \leq CK\left(\sqrt{\frac{t}{n}}\lor \frac{t}{n} \right)
            \end{equation}
            where $C$ is an absolute constant.
        \end{proposition}

   \subsection{Proofs of upper bounds in the setting of heterogeneous missingness}
    \label{proof:heterogeneous}
    
        We denote by $\text{off}(A)$ the matrix obtained by putting to $0$ the diagonal entries of matrix $A$. 
        
        Following \cite{louniciHighdimensionalCovarianceMatrix2014}, we first note that
        \begin{equation}
           \norm{\widehat{\Sigma} - \Sigma}\leq  \norm{\Delta_{\text{inv}} \odot \text{off}\left(\widehat{\Sigma}^Y - \Sigma^Y\right)} + \norm{\diag (\delta_{\text{inv}}) \odot \left(\widehat{\Sigma}^Y - \Sigma^Y\right)}.
        \end{equation}
Hence, in view of \citep[Theorem 3.1.d, page 95]{johnsonMatrixTheory1989}, we have that
\begin{equation}
\label{eq:decomp-first}
           \norm{\widehat{\Sigma} - \Sigma}\leq  \ubar{\delta}^{-2}\norm{\text{off} \left(\widehat{\Sigma}^Y - \Sigma^Y\right)} + \ubar{\delta}^{-1} \norm{\diag(\widehat{\Sigma}^Y - \Sigma^Y)}.
\end{equation}


        We now extend several arguments in \cite{louniciHighdimensionalCovarianceMatrix2014,klochkovUniformHansonWrightType2019} developed in the MCAR setting (same observation rate $\delta$ for all the features) to the heterogeneous missingness setting where each feature $j$ has possibly a different observation rate $\delta_j$ from the others features. 
        

        \begin{lemma}
            Let $X \in \real^p$ be a random vector admitting covariance $\Sigma$. Define $Y^{(j)} = d^{(j)}X^{(j)}$ for all $j \in [p]$, where the $d^{(j)}$ are independent Bernoulli random variables with $\expectation d^{(j)} = \delta_j$ and $\bar{\delta} = \max_j \delta_j$. We have 
            $$\norm{\expectation \left[ (YY^\top - \diag (YY^\top))^2\right]} \lesssim \bar{\delta}^2 \trace{\Sigma} \norm{\Sigma}, $$
            and 
            $$\norm{\expectation \left[ (\diag (YY^\top))^2\right]} \lesssim \bar{\delta} \norm{ \Sigma}^2. $$
        \end{lemma}
        \begin{proof}
            Let us first look at $\expectation \left[ (YY^\top - \diag (YY^\top))^2\right]$. We define $\sqrt{\delta} = (\sqrt{\delta_1} \dots, \sqrt{\delta_p})^\top$. We also denote by $\mathbb{E}_{d}$ the conditional expectation with respect to $d$ given $X$. We compute the following representation:
            \begin{equation*}
            \begin{split}
                \expectation_{d} \text{off} \left(YY^\top\right) ^2 & = \expectation_{d}\text{off} \left((d \otimes d) \odot (X \otimes X)\right)^2\\
                & = \norm{\sqrt{\delta} \odot X}_2^2 \text{off} \left((\delta \otimes \delta) \odot (X \otimes X)\right)\\ 
                & \hspace{0.5cm} - \left(\diag \left((\sqrt{\delta} \otimes \sqrt{\delta}) \odot (X \otimes X)\right) \right) \left( \text{off} \left((\delta \otimes \delta) \odot (X \otimes X)\right) \right)\\
                & \hspace{0.5cm} - \left( \text{off} \left((\delta \otimes \delta) \odot (X \otimes X)\right) \right)\left(\diag \left((\sqrt{\delta} \otimes \sqrt{\delta}) \odot (X \otimes X)\right) \right).
            \end{split}
            \end{equation*}
            Let us name the following matrices:
            \begin{itemize}
                \item $A = \left((\delta \otimes \delta) \odot (X \otimes X)\right)$\\
                \item $B = ((\sqrt{\delta} \otimes \sqrt{\delta}) \odot (X \otimes X))$\\
                \item $C = \text{off} \left((\delta \otimes \delta) \odot (X \otimes X)\right)$\\
                \item $D = \diag \left((\sqrt{\delta} \otimes \sqrt{\delta}) \odot (X \otimes X)\right)$
            \end{itemize}
            We note first that $\text{off}(A) = A - \diag( A) \leq A$ since $\diag ( A)$ is positive semi-definite. Here this inequality is to be understood in the matrix sense (Let $A$ and $B$ be $p\times p $ symmetric matrices. We say that $A\leq B$ if for any $u\in \mathbb{R}^p$, $u^{\top}A u \leq u^{\top}B u$).
            
             Notice also that $-CD - DC \leq (C - D)^2$. Finally, we have that $C-D = A-\diag ( A + B)$. Hence:
            \begin{equation}
            \begin{split}
                \expectation_X \text{off} \left(YY^\top\right) ^2 & \leq \norm{\sqrt{d} \odot X}_2^2 A +  \left(A - \diag(A+B)\right)^2\\
                & \leq \norm{\sqrt{\delta} \odot X}_2^2 A + 2 \left(\diag(A+B)\right)^2 + 2 A^2.
            \end{split}
            \end{equation}

            Let us now compute the expectations according to $X$. Following \cite{louniciHighdimensionalCovarianceMatrix2014}, we find that:
            \begin{equation}
                \expectation A^2 \leq\trace{(\delta \otimes \delta)\odot \Sigma} \left( (\delta \otimes \delta) \odot \Sigma\right) \leq \bar{\delta}^2 \trace{\Sigma} \left[ (\delta \otimes \delta) \odot \Sigma\right].
            \end{equation}
            By elementary computations 
            $$
            \norm{(\delta \otimes \delta) \odot \Sigma} = \max_{u\,:\,\norm{u}=1}\{ \langle (\delta \otimes \delta) \odot \Sigma u,u\rangle  \}= \max_{u\,:\,\norm{u}=1}\{ \langle  \Sigma (\delta\odot u),(\delta\odot u)\rangle  \}.
            $$
            Hence $\norm{(\delta \otimes \delta) \odot \Sigma}\leq \bar{\delta}^2\norm{\Sigma}$ and
            \begin{equation}
                \norm{\expectation A^2} \lesssim \bar{\delta}^4 \trace{\Sigma} \norm{\Sigma}.
            \end{equation} 
            Next, we tackle the diagonal matrix. By the equivalence of the moment of sub-Gaussian distributions we have that $\expectation[(X^{(j)})^4]\lesssim \Sigma_{jj}^2\leq \norm{\Sigma}$. Thus we get that:
            \begin{equation*}
            \begin{split}
                \norm{\expectation \left(\diag(A)\right)^2} 
                &= \norm{\expectation \diag \left( \delta_1^4 \expectation[(X^{(1)})^4],\ldots, \delta_p^4 \expectation[(X^{(p)})^4] \right) } \lesssim  \bar{\delta}^4 \norm{\Sigma}^2.
            \end{split}
            \end{equation*}

            Finally, let us compute $\norm{\expectation \left[ \norm{\sqrt{d} \odot X}_2^2 A \right]}$. For any $u \in \real^p$ such that $\vert u \vert_2 = 1$, we have by Cauchy-Schwartz
            \begin{equation*}
            \begin{split}
                \expectation  \left[ \norm{\sqrt{\delta} \odot X}_2^2 \langle A u, u\rangle\right] & = \expectation\left[\norm{\sqrt{\delta} \odot X}_2^2 \scalarproduct{X \odot \delta}{u}^2\right]\\
                & \leq \left( \expectation \left[\norm{\sqrt{\delta} \odot X}_2^4 \right] \expectation \left[\scalarproduct{X \odot \delta}{u}^4\right] \right)^{\frac{1}{2}}.
            \end{split}
            \end{equation*}
            Looking at the first expectation:
            \begin{equation*}
            \begin{split}
                \expectation \left[\norm{\sqrt{\delta} \odot X}_2^4 \right] & = \expectation \left(\sum_{j\in [p]} \delta_j (X^{(j)})^2\right)^2 \leq \bar{\delta}^2 \expectation \left(\sum_{j\in [p]} (X^{(j)})^2\right)^2 = \bar{\delta}^2\expectation  \norm{X}_2^4.
            \end{split}
            \end{equation*}
Exploiting the sub-Gaussianity of $X$, it was proved in \cite{louniciHighdimensionalCovarianceMatrix2014} that
$$
\expectation  \norm{X}_2^4\lesssim (\mathrm{tr}(\Sigma))^2.
$$
        
We study now the term $\expectation \scalarproduct{X \odot \delta}{u}^4$. Again by sub-Gaussianity of $X$, we have equivalence of the moments. That is for any unit vector $u\in \mathbb{R}^p$ 
            \begin{equation*}
            \begin{split}
                \expectation \scalarproduct{X \odot \delta}{u}^4  &= \expectation \scalarproduct{X }{u\odot \delta}^4 \\&\lesssim \left(\expectation \scalarproduct{X }{u\odot \delta}^2 \right)^2= \langle \Sigma u \odot \delta,u \odot \delta \rangle^2\leq \norm{\Sigma}^2 \vert u \odot \delta  \vert_2^4  \leq \bar{\delta}^4 \norm{\Sigma}^2.
            \end{split}
            \end{equation*}
Combining the last four displays, we obtain that 
\begin{equation}
    \norm{\expectation \left[ \norm{\sqrt{d} \odot X}_2^2 A \right]}\lesssim \bar{\delta}^3 \mathrm{tr}(\Sigma)\norm{\Sigma}.
\end{equation}

The second inequality of the lemma follows from a similar and actually simpler argument. We have
        \begin{equation*}
        \begin{split}
            \norm{\expectation \left[\diag (YY^\top)^2\right]}  = \max_{j} \expectation \left[ d^{(j)} (X^{(j)})^4\right]\lesssim \bar{\delta} \norm{\Sigma}^2,
        \end{split}
        \end{equation*}
        where we have used again the sub-Gaussianity of the random vector $X$ and the fact that $\Sigma_{jj}\leq \norm{\Sigma}$ for any $j\in [p]$.
        \end{proof}

        \begin{lemma}
            Under the same assumptions as the previous lemma, for any $u\in \real^p$ such that $\vert u \vert_2 =1$,
            $$ \expectation \left(u^\top (YY^\top - \diag (YY^\top)) u \right)^2 \lesssim \bar{\delta}^2 \norm{\Sigma}^2,$$
            and
            $$\expectation \left(u^\top \, \diag (YY^\top) u \right)^2 \lesssim \bar{\delta} \norm{\Sigma} .$$
        \end{lemma}
        \begin{proof}
            Let us consider two vectors $\bm{a}$ and $\bm{b}$ of $\real^p$. First, let us demonstrate the first assertion. By using $\expectation Z^2 = \expectation (Z - \expectation Z)^2 + (\expectation Z)^2$ :
            \begin{equation*}
            \begin{split}
                \expectation \left( \sum_{i\neq j} d^{(i)}d^{(j)}a_i b_j \right)= \expectation \left( \sum_{i\neq j} (d^{(i)} - \delta_i)(d^{(j)}-\delta_j)a_i b_j \right) + \left( \sum_{i\neq j} \delta_i \delta_j a_i b_j \right).
            \end{split}
            \end{equation*}
            Looking at the first term, we use the decoupling principle of \cite[Theorem 6.1.1]{vershyninHighDimensionalProbabilityIntroduction2018} to create $d'$ an independent copy of $d$ with same law such that:
            \begin{equation*}
            \begin{split}
               \expectation \left( \sum_{i\neq j} (d^{(i)} - \delta_i)(d^{(j)}-\delta_j)a_i b_j \right) & \leq 16 \expectation \left( \sum_{i\neq j} (d^{(i)} - \delta_i)(d'^{(j)}-\delta_j)a_i b_j \right)\\
                & \leq 16 \sum_{i\neq j} \sum_{k\,eq l} \expectation (d^{(i)} - \delta_i)(d'^{(j)}-\delta_j)(d^{(k)} - \delta_i)(d'^{(l)}-\delta_j)a_i b_j a_k b_l.
            \end{split}
            \end{equation*}
            Here all the terms except $k=i$ and $l=j$ are equal to zero. We have that for all $j$, $\expectation (d^{(i)}- \delta_j)^2 = \delta_j(1-\delta_j) \leq \bar{\delta}$  thus:
            \begin{equation*}
            \begin{split}
               \expectation \left( \sum_{i\neq j} (d^{(i)} - \delta_i)(d^{(j)}-\delta_j)a_i b_j \right) & \leq 16 \sum_{i\neq j}a_i^2b_j^2\delta_i(1-\delta_i)\delta_j(1-\delta_j)\\
                & 16 \sum_{i\neq j}a_i^2b_j^2\bar{\delta}^2 \leq 16 \bar{\delta}^2\vert \bm{a}\vert_2^2 \vert \bm{b}\vert_2^2.
            \end{split}
            \end{equation*}
            The second term can be bounded using $(x + y)^2 \leq 2x^2 + 2y^2$ and Cauchy-Schwarz:
            \begin{equation*}
            \begin{split}
                 \left( \sum_{i\neq j} \delta_i \delta_j a_i b_j \right) & \leq 2\left(\sum_i \delta_i^2a_ib_i \right)^2 + 2 \left( \sum_{i,j} \delta_i \delta_j a_i b_j \right)^2\\
                & \leq 2\left(\sum_i \delta_i^2 \vert a_ib_i \vert \right)^2 + 2 \left( \sum_{i,j} \delta_i \delta_j \vert a_i b_j \vert \right)^2\\
                & \leq 2 \bar{\delta}^2 \left(\sum_i \vert a_ib_i \vert \right)^2 + 2 \bar{\delta}^4\left( \sum_{i,j} \vert a_i b_j \vert \right)^2.
            \end{split}
            \end{equation*}
            The rest of the proof follows the same arguments to those in the proof of \citep[Lemma 4.4]{klochkovUniformHansonWrightType2019}.

            Regarding the second assertion, it is immediate to see that:
            \begin{equation*}
            \begin{split}
                \expectation \left( u^\top \diag(Y Y^\top) u \right)^2 & = \expectation \left( \sum_j d^{(j)} u_j^2 X^{(j)^2} \right)^2\\
                & = \sum_j \delta_j u_j^4 \expectation (X^{(j)})^4 + \sum_{i\neq j} \delta_i\delta_j u_i^2\expectation \left[(X^{(i)})^2(X^{(j)})^2\right]u_j^2\\
                & \leq \bar{\delta}^2 \norm{\Sigma}^2 + (\bar{\delta} - \bar{\delta}^2)\sum_j u_j^4 \norm{\Sigma}^2 \lesssim \bar{\delta}\norm{\Sigma}^2 .              
            \end{split}
            \end{equation*}
        \end{proof}

We can now apply the Bernstein inequality \citep[Proposition 4.1]{klochkovUniformHansonWrightType2019} to each term in the right-hand-side of \eqref{eq:decomp-first}.
      
    \subsection{Proof of the upper bounds in the contaminated case}
    \label{sec:proofcontaminated}

\paragraph{Proof of Theorem \ref{th:upper_contaminated}.} First, observe that:
\begin{equation}
\begin{split}
    \norm{\widehat{\Sigma} - \Sigma} & \leq 2*\delta^{-2} \norm{\widehat{\Sigma}^Y - \Sigma^Y} + \frac{\varepsilon(1-\delta)}{\delta} \norm{\Lambda}
\end{split}
\end{equation}
Thus we need to control the error on the observed covariance matrix. Here the error can decomposed as follows:
\begin{equation}
    \norm{\widehat{\Sigma}^Y - \Sigma^Y} \leq \norm{\widehat{\Sigma}^\delta - \Sigma^\delta} + \norm{\widehat{\Lambda}^\varepsilon - \Lambda^\varepsilon} + \norm{\widehat{\Sigma}^{X\xi \delta \varepsilon}}.
\end{equation}
where the three empirical matrices are
\begin{enumerate}
    \item $\widehat{\Sigma}^\delta = n^{-1} \sum_{i=1}^n (d_i \otimes d_i) \odot (X_i \otimes X_i)$, the empirical covariance matrix of the $d_i \odot X_i$ and $\Sigma^\delta= \mathbb{E}\left[\widehat{\Sigma}^\delta  \right]$;
    \item $\widehat{\Lambda}^\varepsilon = n^{-1} \sum_{i=1}^n  \left( [(1-d_i)\odot e_i] \otimes [(1-d_i)\odot e_i]\right) \odot (\xi_i \otimes \xi_i)$, the empirical covariance of the $(1-d_i)\odot e_i\odot \xi_i$ is such that $\Lambda^\varepsilon=\mathbb{E}\widehat{\Lambda}^\varepsilon= \varepsilon(1-\delta)\Lambda$;
    \item $\widehat{\Sigma}^{X,\xi,\delta, \varepsilon} = n^{-1} \sum_{i=1}^n \left(  d_i \otimes [(1-d_i)\odot e_i]\right)\odot(X_i \otimes \xi_i) + ( [(1-d_i)\odot e_i]\otimes  d_i)\odot(\xi_i \otimes X_i)$ is the empirical covariance between the $d_i \odot X_i$ and the $(1-d_i)\odot e_i\odot\xi_i$ and has a null diagonal and also a null expectation.
\end{enumerate}
 Using \cite{klochkovUniformHansonWrightType2019}, we get that there exists an absolute constant $C>0$ such that, with probability at least $1-e^{-t}$,
\begin{equation}
    \label{eq:proof_contaminated_proof2}
    \norm{\widehat{\Sigma}^\delta - \Sigma^\delta } \leq C \delta \norm{\Sigma}\left(\sqrt{\frac{\bm{r}(\Sigma) \log\bm{r}(\Sigma)}{n}} \lor \sqrt{\frac{t}{n}} \lor \frac{\bm{r}(\Sigma)(t+\log\bm{r}(\Sigma))}{ \delta n}\log(n)\right)
\end{equation}

To tackle the second term, we will use a standard argument for isotropic sub-Gaussian random vectors (see for instance the proof of Theorem 5.39 in \cite{vershyninIntroductionNonasymptoticAnalysis2011}) combining a vector Bernstein inequality \citep[Corollary 5.17]{vershyninIntroductionNonasymptoticAnalysis2011} with an union bound. Hence we obtain with probability at least $1-e^{-t}$
\begin{equation}
    \norm{ \widehat{\Lambda}^\varepsilon - \expectation \widehat{\Lambda}^\varepsilon } \lesssim (1-\delta)\varepsilon \sigma_{\xi}^2 \psinorm{(1-d)\odot e \odot \xi}{1}\left( \sqrt{\frac{p}{n}} \vee \frac{p}{n} \vee \sqrt{\frac{t}{n}} \vee \frac{t}{n} \right).
\end{equation}
Set $d'=(1-d)\odot e $ and $\varepsilon'=(1-\delta)\varepsilon$. We note that $d'_j\sim B(\varepsilon')$ for any $j\in [p]$.
Using the properties of the Orlicz norm, we easily get that $\psinorm{d'\odot \xi}{1}\leq \psinorm{d'}{2}\psinorm{ \xi}{2}$. Next, by triangular inequality,  we note that
$\psinorm{d'}{2}\leq \psinorm{d'- \delta \mathds{1} }{2}+ \psinorm{\varepsilon' \mathds{1} }{2}\lesssim \psinorm{d- \varepsilon' \mathds{1} }{2}+ \varepsilon'$. Theorem 1.1 in \cite{schlemm2016kearns} guarantees that $\psinorm{d'- \varepsilon' \mathds{1} }{2}\lesssim \sqrt{\frac{1-2\varepsilon' }{4\log((1-\varepsilon')/\varepsilon')}} \lesssim \frac{1}{\sqrt{|\log \varepsilon'|}}$ for any $\varepsilon'<1/4$. Since $\varepsilon' \log \varepsilon'\rightarrow 0 $ as $\varepsilon' \rightarrow 0+$, we obtain that $\psinorm{d'}{2}\lesssim \frac{1}{\sqrt{|\log \varepsilon'|}}$.

Hence the previous display becomes
\begin{equation}
\label{eq:proof_contaminated_proof3}
    \norm{ \widehat{\Lambda}^\varepsilon - (1-\delta)\varepsilon \sigma_{\xi}^2 I_p } \lesssim \frac{(1-\delta)\varepsilon}{\sqrt{|\log ((1-\delta)\varepsilon)}|} \sigma_{\xi}^2 \left( \sqrt{\frac{p}{n}} \vee \frac{p}{n} \vee \sqrt{\frac{t}{n}} \vee \frac{t}{n} \right).
\end{equation}



Now we need to control the norm of 
\begin{equation*}
    \widehat{\Sigma}^{X,\xi,\delta, \varepsilon}  = n^{-1} \sum_{i=1}^n \left(  d_i \otimes [(1-d_i)\odot e_i]\right)\odot(X_i \otimes \xi_i) + ( [(1-d_i)\odot e_i]\otimes  d_i)\odot(\xi_i \otimes X_i).
\end{equation*}
To this end, we apply again the noncommutative Bernstein inequality of \cite[Proposition 4.1]{klochkovUniformHansonWrightType2019}. Note that this result was stated for Hermitian matrices, but the result can be easily extended to arbitrary matrices by applying the self-adjoint dilation trick (See for instance \cite{troppIntroductionMatrixConcentration2015} for more details).

In what follows, for any $i\in [n]$, we set
$$
Z_i:=\left(d_i\otimes [(1-d_i)\odot e_i]\right)\odot(X_i \otimes \xi_i) + ( [(1-d_i)\odot e_i]\otimes  d_i)\odot(\xi_i \otimes X_i).
$$
For the sake of simplicity, we will write $Z$ without any index $i$ to designate any of the $Z_i$'s. Notice that $Z$ is symmetric by construction but has no diagonal term.

\begin{lemma}\label{lem:VarZi} Under the assumptions of Theorem \ref{th:upper_contaminated}, we have
    $$
    \norm{\expectation Z^\top Z} \leq 
    \delta^2(1-\delta)\varepsilon (p-2) \sigma_{\xi}^2\left[2 \norm{\Sigma} + \sigma_{\xi}^2 \right] + \delta(1-\delta) \varepsilon \sigma_{\xi}^4\left( \left| \mathrm{tr}(\Sigma) - \delta (p-2) \right| +\norm{ \Sigma}\right).
    $$

\end{lemma}
\begin{proof}
We first compute the matrix product $Z^\top Z$. For any $k,l\in [p]$ 
\begin{equation}
    \begin{split}
            \left( Z^\top Z\right)_{kl} = & \sum_{j=1}^p Z_{kj}Z_{jl}\\
            = & \sum_{j=1}^p \left(\underbrace{d^{(k)}\left(1-d^{(j)}\right)e^{(j)} X^{(k)} \xi^{(j)}}_{(i)} + \underbrace{d^{(j)}\left(1-d^{(k)}\right)e^{(k)}X^{(j)}\xi^{(k)}}_{(ii)}\right)\\
            & \qquad \times \left(\underbrace{d^{(j)}\left(1-d^{(l)}\right)e^{(l)}X^{(j)}\xi^{(l)}}_{(iii)} + \underbrace{d^{(l)}\left(1-d^{(j)}\right)e^{(j)}X^{(l)}\xi^{(j)}}_{(iv)}\right)
    \end{split}
    \end{equation}
    Let us call $i$ and $ii$ the two terms inside the first factor and $iii$ and $iv$ the two terms in the second factor. Observe that most terms simplify when taking the expectation:
    \begin{itemize}
        \item \textbf{$\bm{(i)}$ times $\bm{(iii)}$:} is always zero. Indeed, if $j \neq l$ then by independence of $\xi^{(j)}$ and $\xi^{(l)}$; otherwise $d^{(l)}(1-d^{(l)}) = 0$.
        \item \textbf{$\bm{(ii)}$ times $\bm{(iii)}$:} only remains if $j \neq k$ or $j\neq l$.
        \item \textbf{$\bm{(i)}$ times $\bm{(iv)}$:} is zero if $k\neq l$ by independence of $\xi^{(k)}$ and $\xi^{(l)}$, otherwise the whole sum remains.
        \item \textbf{$\bm{(ii)}$ with $\bm{(iv)}$:} is always zero if $j\neq k$ by independence of $\xi^{(j)}$ and $\xi^{(k)}$ and if $j=k$ then $d^{(k)}(1-d^{(k)}) = 0$.
    \end{itemize}
    We can thus rewrite:
    \begin{equation*}
        \expectation \left(Z^\top Z\right)_{kl} = \begin{cases}
            \sum_{j=1}^p \delta^2(1-\delta)\varepsilon \expectation \left(X^{(l)}X^{(k)} (\xi^{(j)})^2\right) \quad \text{if $k\neq l$}\\
            \sum_{j=1}^p \delta (1-\delta)\varepsilon \expectation \left( (X^{(j)})^2 (\xi^{(k)})^2\right) \quad \text{if $k=l$}
        \end{cases}
    \end{equation*}
    By computing the expectations using the independence of our variables, we get:
    \begin{equation*}
    \begin{split}
        \expectation \left(Z^\top Z\right)_{kl} & = \begin{cases}
            \delta^2(1-\delta)\varepsilon\Sigma_{kl} \sum_{j\in [p]\setminus \lbrace k,l\rbrace} \Lambda_{jj}, \quad \text{if $k\neq l$,}\\
            \delta(1-\delta)\varepsilon\Lambda_{k} \sum_{j\in [p]\setminus \lbrace k,l\rbrace} \Sigma_{jj}, \quad \text{if $k = l$.}
        \end{cases}\\
    \end{split}
    \end{equation*}
Thus, for $A$ the matrix such that:
    \begin{equation*}
        A_{kl} = \delta(1-\delta)\varepsilon\begin{cases}
            \delta \sum_{j\in [p]\setminus \lbrace k,l\rbrace} \Lambda_{jj}, \quad \text{if $k\neq l$,}\\
            \sum_{j\in [p]\setminus \lbrace k \rbrace} \Sigma_{jj}, \quad \text{if $k = l$.}
        \end{cases}
    \end{equation*}
    We can write the expectation as
    \begin{equation}
    \label{eq:ExpectZZT}
        \expectation \left(Z^\top Z\right) = \left(\Sigma - \diag{\Sigma} + \Lambda \right) \odot A.
    \end{equation}
With our i.i.d. assumptions on the $\xi_i^{(j)}$'s contaminations, we can simplify the previous expression of $A$
 \begin{equation*}
        A_{kl} = \delta(1-\delta)\varepsilon  \sigma_{\xi}^2
        \begin{cases}
            \delta (p-2), \quad \text{if $k\neq l$,}\\
            \sum_{j\in [p]\setminus \lbrace k \rbrace} \Sigma_{jj}, \quad \text{if $k = l$.}
        \end{cases}
    \end{equation*}

Denote by $J$ the $p\times p$ matrix with all its entries equal to $1$. Then we have the following equivalent representation for $A$
\begin{equation}
A= \delta^2(1-\delta)\varepsilon (p-2) \sigma_{\xi}^2 J + \delta(1-\delta) \varepsilon \sigma_{\xi}^2\left( \left[ \mathrm{tr}(\Sigma) - \delta (p-2) \right] I_p - \mathrm{diag}(\Sigma)\right).
\end{equation}
We deduce from the previous display and \eqref{eq:ExpectZZT} that
\begin{equation*}
\expectation \left(Z^\top Z\right) = \delta^2(1-\delta)\varepsilon (p-2) \sigma_{\xi}^2\left[\Sigma-\mathrm{diag}(\Sigma)+\sigma_{\xi}^2 I_p\right] + \delta(1-\delta) \varepsilon \sigma_{\xi}^4\left( \left[ \mathrm{tr}(\Sigma) - \delta (p-2) \right] I_p - \mathrm{diag}(\Sigma)\right).
\end{equation*}

Hence, using the Schur-Horn theorem \cite{devadasSchurHornTheorem2015}(a direct consequence of which is that $\norm{\mathrm{diag}(\Sigma)}\leq \norm{\Sigma}$), we get
\begin{equation*}
\norm{\expectation \left(Z^\top Z\right) }\leq \delta^2(1-\delta)\varepsilon (p-2) \sigma_{\xi}^2 \left[2 \norm{\Sigma} + \sigma_{\xi}^2 \right] + \delta(1-\delta) \varepsilon \sigma_{\xi}^4\left( \left| \mathrm{tr}(\Sigma) - \delta (p-2) \right| +\norm{ \Sigma}\right).
\end{equation*}
\end{proof}

\begin{lemma}\label{lem:maxZi}
    We can bound the psi-1 norm of the maximum of the $Z_i$ as follows:
    \begin{equation}
        \psinorm{\max_{i\in [n]} \norm{Z_i}}{1} \lesssim \sqrt{\delta (1-\delta)\varepsilon} \sqrt{\norm{\Sigma}\norm{\Lambda} \bm{r}(\Sigma)\, \bm{r}(\Lambda)} \log(n).
    \end{equation}
\end{lemma}
\begin{proof}
    Indeed, as recalled in \cite[Remark 4.1]{klochkovUniformHansonWrightType2019},
    \begin{equation*}
        \psinorm{\max_{i\in [n]} \norm{Z_i}}{1} \lesssim \log n \, \max_{i\in [n]} \psinorm{\norm{Z_i}}{1} 
    \end{equation*}

    By definition of the spectral norm:
    \begin{equation}
    \begin{split}
        \norm{Z} & = \max_{\norm{u}_2 \leq 1} \lbrace u^\top Z u \rbrace\\
        & = \max_{\norm{u}_2 \leq 1} 2 \scalarproduct{d\odot X}{u}\scalarproduct{e\odot(1-d)\odot \xi}{u}\\
        & \leq 2 \left( \max_{\norm{u}_2 \leq 1}\scalarproduct{d\odot X}{u}\right) \left( \max_{\norm{u}_2 \leq 1}\scalarproduct{e\odot(1-d)\odot \xi}{u}\right)\\
         & = 2 \norm{d\odot X}_2 \norm{e\odot(1-d)\odot \xi}_2.
    \end{split}
    \end{equation}
     Then, we can see that $\norm{Z}$ is sub-exponential
    \begin{equation}
    \begin{split}
        \psinorm{\norm{Z}}{1} & \leq 2 \psinorm{\norm{d\odot X}_2}{2} \psinorm{\norm{(1-d)\odot e \odot \xi}_2}{2}.
    \end{split}
    \end{equation}   
    Since 
    $\norm{d \odot X}$ is sub-Gaussian, we get
    \begin{equation}
    \begin{split}
        \psinorm{\norm{d\odot X}_2}{2} & \lesssim \sqrt{\expectation \left[\norm{d \odot X}_2^2\right]} = \sqrt{\delta \trace{\Sigma}}.
    \end{split}
    \end{equation}
    Similarly, we get that:
    \begin{equation}
    \begin{split}
        \psinorm{\norm{(1-d)\odot e \odot \xi}_2}{2} & \lesssim \sqrt{\expectation \left[\norm{(1-d)\odot e \odot \xi}_2^2\right]} = \sqrt{(1-\delta)\varepsilon \expectation \left[ \norm{ \xi}_2^2\right]}\lesssim \sqrt{(1-\delta)\varepsilon \,\trace{\Lambda}}.
    \end{split}
    \end{equation}
    Which in turn gives us that
    \begin{equation}
        \max_{i\in [n]} \psinorm{\norm{Z_i}}{1} \lesssim \sqrt{\delta (1-\delta)\varepsilon} \sqrt{\norm{\Sigma}\norm{\Lambda} \bm{r}(\Sigma)\, \bm{r}(\Lambda)}.
    \end{equation}
\end{proof}

Finally, combining Lemmas \ref{lem:VarZi} and \ref{lem:maxZi} (with $\bm{r}(\Lambda)=p$) with Theorem 4.1 of \cite{klochkovUniformHansonWrightType2019}, we get, with probability at least $1-e^{-t}$,
    \begin{align*}
&\norm{n^{-1}\sum_{i=1}^n Z_i}\\
&\hspace{0.5cm}\lesssim \sqrt{\delta^2(1-\delta)\varepsilon (p-2)\sigma_{\xi}^2 \left[2 \norm{\Sigma} + \sigma_{\xi}^2 \right] + \delta(1-\delta) \varepsilon \sigma_{\xi}^4\left( \left| \mathrm{tr}(\Sigma) - \delta (p-2) \right| +\norm{ \Sigma}\right)}  \sqrt{\frac{t+\log (p)}{n}} \\
        &\hspace{1cm}+ \sqrt{\delta (1-\delta)\varepsilon\,\sigma_{\xi}^2\, p} \sqrt{ \mathrm{tr}(\Sigma)} \log(n) \frac{t+\log (p)}{n}.
\end{align*}

We consider the case $\delta\,(p-2)\geq \mathrm{tr}(\Sigma)$ and $\sigma^2_{\xi}\geq \norm{\Sigma}$.
Then the previous display becomes, with probability at least $1-e^{-t}$,
 \begin{align}
    \label{eq:proof_contaminated_proofbis}
\norm{n^{-1}\sum_{i=1}^n Z_i}&\lesssim \sqrt{\delta^2(1-\delta)\varepsilon \, \sigma_{\xi}^4\, p}  \sqrt{\frac{t+\log (p)}{n}}+ \sqrt{\delta (1-\delta)\varepsilon\,\sigma_{\xi}^2\, p} \sqrt{ \mathrm{tr}(\Sigma)}  \log(n) \frac{t+\log (p)}{n}\notag\\
&\lesssim \sqrt{\delta(1-\delta)\varepsilon  \sigma_{\xi}^2\, p} \sqrt{\frac{t+\log (p)}{n}}\biggl( 
\sqrt{\delta \, \sigma_\xi^2} + \sqrt{ \mathrm{tr}(\Sigma)}  \log(n) \sqrt{\frac{t+\log (p)}{n}} \biggr).
\end{align}
An union bound combining the previous display with \eqref{eq:proof_contaminated_proof2} and \eqref{eq:proof_contaminated_proof3} gives the result, up to a rescaling of the constants,
with probability at least $1-e^{-t}$.
        
    \section{Proof of lower bounds}
    \label{app:lower_bound}

        The first two subsections deal with the lower bound of theorem \ref{th:lower}, the third extends it to the contaminated case.
        
        \subsection{Hypothesis construction in the Grassmannian manifold}

            Let $H$ be a $p\times r$ matrix with orthonormal rows. Each matrix $H$ describes a subspace $U_H$ of $\real^p$, where $\dim(U_H) = r$ and $H^\top H$ is its projector in $\real^p$. The set of all $U_H$ is the Grassmannian manifold $G_r(\real^p)$, which is the set of all $r$-dimensional subspaces of $\real^p$. The Grassmannian manifold is a smooth manifold of dimension $d = r(p-r)$, where one can define a metric for all subspaces $U, \bar{U} \in G_r(\real^p)$:
            \begin{equation}
                d(U,\bar{U}) = \norm{P_U - P_{\bar{U}}}_F = \norm{H^\top H - \bar{H}^\top \bar{H}}_F
            \end{equation}
            where $P_U$ and $P_{\bar{U}}$ are the projectors to the subspaces $U$ and $\bar{U}$ respectively and $H$ and $\bar{H}$ are the $r\times p$ matrix with orthonormal rows associated with $U$ and $\bar{U}$ respectively. In the remainder of the proof, we will identify the projectors to the subspaces. A result on the entropy of Grassmanian manifolds \citep{pajorMetricEntropyGrassmann1998} shows that:
            
            \begin{proposition}
            \label{prop:entropy}
                For all $\varepsilon > 0$, there exists a family of orthonormal projectors $\mathcal{U} \subset G_r(\real^p)$ such that:
                \begin{equation}
                    \vert \mathcal{U} \vert \geq \left\lfloor \frac{\bar{c}}{\varepsilon}\right\rfloor^d,
                \end{equation}
                and, $\forall P,Q \in G_r(\real^p), P \neq Q$,
                \begin{equation}
                    \bar{c} \varepsilon \sqrt{r} \leq \norm{P - Q}_F \leq \frac{\varepsilon \sqrt{r}}{\bar{c}}.
                \end{equation}
                for some small enough absolute constant $\bar{c}$, where $\vert\mathcal{U}\vert$ is the cardinal of set $\mathcal{U}$.
            \end{proposition}
            Without loss of generality, we assume that the block matrix $P_1 = \begin{pmatrix}I_r & 0 \\ 0 & 0\end{pmatrix}$ belongs to the set $\mathcal{U}$. Indeed, the Frobenius norm is invariant through a change of basis. 
            
            Let us then build such a set $\mathcal{U}$ of hypotheses. Let $\gamma = a \sqrt{\nicefrac{p}{\delta^2 n}}$ where $a>0$ is an absolute constant 
            We set $N = \vert \mathcal{U}\vert$ and $\mathcal{U} = \lbrace P_1, \dots, P_N\rbrace$ where $P_1$ was introduced above. 
            Let us define the family of $p\times p$ symmetric matrices $\Sigma_1, \dots, \Sigma_N$ , $\forall j \in \lbrace 1, N\rbrace$ as follows : $\Sigma_j = I_p + \gamma P_j$, where $I_p$ is the $p\times p$ identity matrix. These covariance matrices belongs to the class of spiked covariance matrices.
            
            Then, we can see that, for $i,j \in \lbrace 1, \dots N \rbrace$, by setting $\varepsilon = \nicefrac{1}{2}$:
            \begin{equation}
            \label{eqn:min_dist_frobenius}
                \norm{\Sigma_i - \Sigma_j}_F^2 = \gamma^2 \norm{P_i - P_j}_F^2 > a^2\bar{c}^2\frac{pr}{2\delta^2 n}
            \end{equation}
            
        \subsection{KL-divergence of hypotheses}
        
            Now that we have our candidate covariances $\Sigma_1, \dots, \Sigma_N$, let us define the associated distributions. For $j \in \lbrace 1, N\rbrace$, let $X_1, \dots X_n$ be i.i.d. random variables following a gaussian $\mathcal{N}(0, \Sigma_j)$ law. Let $d_1, \dots d_n$ be each vectors of $p$ i.i.d bernoulli random variables of probability of success $\delta > 0$, and let $Y_1, \dots Y_n$ be random variables such that, $\forall i \in \lbrace 1, n \rbrace, Y_i = d_i \odot X_i$, with $\odot$ the Hadamard or term-by-term product. Let us also define as $\mathbb{P}_j$ the distribution of $Y_1, \dots Y_n$ and $\mathbb{P}_j^{(\delta)}$ the conditional distribution of the $Y_1, \dots Y_n$ knowing $d_1, \dots d_n$. Finally, let $\expectation_j$ be the expectation given the distribution associated with the $j$-th projector and $\expectation_d$ the expectation over $d_1, \dots d_n$.
            
            For $j\in \lbrace 2, \dots, N\rbrace$, let us compute the Kullback-Leibler divergence from $\mathbb{P}_1$ to $\mathbb{P}_j$.
            \begin{equation}
            \begin{split}
                \text{KL}(\mathbb{P}_1, \mathbb{P}_j) & = \expectation_1 \log\left(\frac{d\mathbb{P}_1}{d\mathbb{P}_j}\right) = \expectation_1 \log\left(\frac{d\mathbb{P}_j^{(\delta)} \otimes\mathbb{P}_1^{(\delta)}}{d\mathbb{P}_j^{(\delta)} \otimes\mathbb{P}_j^{(\delta)}}\right)\\
                & = \expectation_d \text{KL}(\mathbb{P}_1^{(\delta)}, \mathbb{P}_j^{(\delta)}) = \sum_{i=1}^n \expectation_d \text{KL}(\mathbb{P}_1^{(d_i)}, \mathbb{P}_j^{(d_i)}).
            \end{split}
            \end{equation}
            Since $\forall i \in \lbrace 1, \dots, n\rbrace$, $Y_i\vert d_i \sim \mathcal{N}\left(0, (d_i \otimes d_i)\odot \Sigma\right)$, for all $j\in \lbrace 1, \dots N \rbrace$ and for each realisation $\delta(\omega) \in \lbrace 0, 1 \rbrace^p$, $\mathbb{P}_j \gg \mathbb{P}_1$, thus $\text{KL}(\mathbb{P}_1, \mathbb{P}_j) < \infty$.
            
            Define $J_i = \lbrace j: d_{i,j} = 1, 1 \leq j \leq p \rbrace$ the set of indices kept by vector $d_i$ and $p_i = \sum_{j=1}^p d_{i,j} \sim \mathcal{B}(p,\delta)$. Then, define the mapping $Q_i : \mathbb{R}^p \rightarrow \mathbb{R}^{d_i}$ such that $Q_i(x) = x_{J_i}$, such that $x_{J_i}$ is a $p_i$ dimensional vector containing the components of $x$ whose index are in $J_i$. Let $Q_i^* : \mathbb{R}^{d_i} \rightarrow \mathbb{R}^p$ the right inverse of $Q_i$.
            
            Note that $\forall j \in \lbrace 1, N-1\rbrace$, $\Sigma_j = (1+\gamma) P_j + P_j^\perp$, with $P_j^\perp$ the projector to the subspace of $\mathbb{R}^p$ orthogonal to the one described by $P_j$. Let us define $\Sigma_j^{(d_i)} = Q_i \Sigma_j Q_i^*$. Then, observe that $\Sigma_1^{(d_i)}$ is invertible, with inverse $Q_i \left(\frac{1}{\gamma+1} P_1 + P_1^\perp\right)Q_i^*$ since $P_1$ and $P_1^\perp$ are diagonal matrices. We thus get, for $i\in \lbrace 1, \dots n \rbrace$:
            \begin{equation}
                \text{KL}(\mathbb{P}_1^{(d_i)}, \mathbb{P}_j^{(d_i)}) = \frac{1}{2} \left( \trace{\Sigma_1^{(d_i)^{-1}}\Sigma_j^{(d_i)} } - p_i - \log(\det (\Sigma_1^{(d_i)^{-1}}\Sigma_j^{(d_i)})) \right).
            \end{equation}
            
First, using a result of linear algebra described in section \ref{proof:determinant}, we show that:
\begin{equation}
        \begin{split}
        - \expectation_d \log(\det (\Sigma_1^{(d_i)^{-1}}\Sigma_j^{(d_i)})) \leq a r \sqrt{p/n}.
        \end{split}
\end{equation}
In the high-dimensional regime $p\geq n$, we obtain 
\begin{equation}
        \label{eqn:pinsker}
        \begin{split}
        - n \;  \expectation_d \log(\det (\Sigma_1^{(d_i)^{-1}}\Sigma_j^{(d_i)})) \leq a r \sqrt{n \, p} \leq a\, r\, p.
        \end{split}
\end{equation}
            
            Next, let us focus on bounding  $\frac{1}{2} \trace{\Sigma_1^{(d_i)^{-1}} (\Sigma_j^{(d_i)} - \Sigma_1^{(d_i)})}$. Remember that $\Sigma_1$ is diagonal. Using the fact that $\Sigma_1^{-1} = \frac{1}{1+\gamma} P_1 + P_1^\perp$, we get:
            \begin{equation}
            \begin{split}
                \trace{\Sigma_1^{(d_i)^{-1}} (\Sigma_j^{(d_i)} - \Sigma_1^{(d_i)})} & = \frac{\gamma}{1+\gamma} \trace{Q_i P_1 ( P_j - P_1) Q_i^*} + \gamma \trace{Q_i P_1^\perp (P_j - P_1) Q_i^*}\\
                & = \frac{\gamma}{1+\gamma} \left(\trace{Q_iP_1P_jQ_i^*} - \trace{Q_iP_1Q_i^*}\right)+ \gamma\trace{Q_i\left(I_p - P_1\right) P_j Q_i^*}\\
                & = \left(\frac{\gamma}{1+\gamma} - \gamma \right) \left( \trace{Q_i P_1 P_j Q_i^*} - p_i\right)\\
                & = \frac{\gamma^2}{2(1+\gamma)} \norm{Q_i(P_j - P_1)Q_i^*}_F^2
            \end{split}
            \end{equation}
            Finally, using the fact demonstrated in appendix \ref{proof:frobenius} and the upper bound of proposition \ref{prop:entropy}, we get that:
            \begin{equation}
            \begin{split}
                \text{KL}(\mathbb{P}_1, \mathbb{P}_j) & \leq \sum_{i=1}^n\expectation_d \frac{\gamma^2}{2(1+\gamma)} \norm{Q_i(P_j - P_1)Q_i^*}_F^2\\
                & \leq \sum_{i=1}^n \frac{\gamma^2\delta}{2(1+\gamma)} \norm{P_j - P_1}_F^2\\
                & \leq \sum_{i=1}^n \frac{\gamma \delta r }{8\bar{c}^2}\leq \frac{a}{8\bar{c}^2 } r \sqrt{p\, n} \leq
                \frac{a^2}{4 \bar{c}^2}  r\, p .
            \end{split}
            \end{equation}
            Thus, since $N \geq \lfloor 2 \bar{c} \rfloor ^{r(p-r)}$, and since we assumed that $p > 2r$:
            \begin{equation}
                \text{KL}(\mathbb{P}_1, \mathbb{P}_j) \leq \alpha\log (N),
            \end{equation}
            for $\alpha = \nicefrac{a^2}{8\bar{c}^2}$. According to theorem 2.5 of \cite{tsybakovNonparametricEstimators2009}, the previous display combined with \eqref{eqn:min_dist_frobenius} gives
            \begin{equation}
            \inf_{\widehat{\Sigma}} \sup_{\mathbb{P}_\Sigma} \mathbb{P}_\Sigma \left( \norm{\widehat{\Sigma} - \Sigma}_F^2 \geq C \frac{r}{\delta^2 n}p \right) \geq \beta,
            \end{equation}
            where $C>0$ and $\beta>0$ are two absolute constants. This fact, in turn, implies the lower bound of theorem \ref{th:lower}, since, for all $\Sigma_1, \Sigma_2$ matrices of our hypothesis set:
            \begin{equation}
                \norm{\Sigma_1-\Sigma_2}^2 \geq C \frac{r}{\delta^2 n}.
            \end{equation}
            Indeed, otherwise, we would get
            \begin{equation}
                \norm{\Sigma_1-\Sigma_2}_F^2 < p\norm{\Sigma_1-\Sigma_2}^2 < C\frac{r}{\delta^2 n}p,
            \end{equation}
            which contradicts equation \ref{eqn:min_dist_frobenius}.

            The heterogeneous result follows immediately by replacing $\delta$ with $\ubar{\delta}$.

    \subsection{Lower bound in the contaminated case}
    \label{sec:proof_lower_bound_contamination}

        The bound of theorem \ref{th:lower_contaminated} is made of two terms. The left term is the missing values lower bound, since missingness is a particular case of contamination. The second term is a result from the Huber contamination analysis of \cite{chenRobustCovarianceScatter2017}, which we develop here.

        The proof is based on Le Cam's two point argument (see e.g. chapter 2.3 of \cite{tsybakovNonparametricEstimators2009}). Let $\Sigma_1 = I_p$ and $\Sigma_2 = I_p + \frac{(1-\delta)\varepsilon}{\delta} E_{11}$ where $E_{11}$ is the matrix with   zeros except in the $(1,1)$ entry, which is equal to $1$. Then, let $P_1 = \mathcal{N}(0, \Sigma_1)$ and $P_2 = \mathcal{N}(0, \Sigma_2)$. We will now build two contaminations $Q_1$ and $Q_2$ such that they render $P_1$ and $P_2$ undistinguishable under cell-wise contamination of parameter $\delta$ and $\varepsilon$. Notice for now that:
        \begin{equation}
            \norm{\Sigma_1 - \Sigma_2} = \frac{\varepsilon (1-\delta)}{\delta}
        \end{equation}
        and by Pinsker's inequality \citep[Lemma 2.5]{tsybakovNonparametricEstimators2009}:
        \begin{equation}
            \text{TV}(P_{1,1}, P_{2,1})^2 \leq \frac{1}{2}\text{KL}\left(P_{1,1}, P_{2,1}\right) \leq \frac{1}{8} \left(1 - 1 - \frac{(1-\delta)\varepsilon}{\delta}\right)^2 = \frac{1}{8}\left(\frac{(1-\delta)\varepsilon}{\delta}\right)^2
        \end{equation}
        and fix $\varepsilon' \leq \frac{1}{\sqrt{8}} \varepsilon \leq \varepsilon$ such that $\text{TV}(P_{1,1}, P_{2,1})^2 = \frac{(1-\delta)\varepsilon'}{\delta}$.
        
        We will create our $Q_1$ and $Q_2$ such that they both have independent components. Since $P_1$ and $P_2$ are isotropic Gaussians and the contamination is completely at random, we can decompose the contaminated distributions $\tilde{P}_1$ and $\tilde{P}_2$ as follows:
        \begin{equation*}
            \tilde{P}_1 = \prod_{i=1}^p \delta P_{1, i} + \varepsilon(1-\delta) Q_{1,i}
        \end{equation*}
        and 
        \begin{equation*}
            \tilde{P}_2 = \prod_{i=1}^p \delta P_{2, i} + \varepsilon(1-\delta) Q_{2,i}
        \end{equation*}
        Notice that taken separately, the components can be considered to be univariate Gaussian distributions under a Huber contamination. We can now try to build $Q_1$ and $Q_2$ so that $\tilde{P}_1$ and $\tilde{P}_2$ are equal in distribution. Let us first set $Q_{1,i} = Q_{2,i} = \mathcal{N}(0,1)$ for $i\neq 1$, since the components of $P_1$ and $P_2$ are equal in distribution for $i\neq 1$ the contamination we choose here doesn't matter much. 
        
        The rest of the proof is heavily inspired by \citep[Appendix E]{chenRobustCovarianceScatter2017}. Set the following densities: 
        \begin{equation*}
            p_1 = \frac{dP_{1,1}}{d(P_{1,1} + P_{2,1})} \qquad \text{and} \qquad p_2 = \frac{dP_{1,1}}{d(P_{2,1} + P_{2,1})}
        \end{equation*}
        Then, define the following contaminations $Q_{1,1}$ and $Q_{2,1}$:
        \begin{equation*}
            \frac{d Q_{1,1}}{d(P_{1,1} + P_{2,1})} = \frac{(p_2 - p_1) \mathbb{I} \lbrace p_2 \geq p_1 \rbrace}{\text{TV}(P_{1,1}, P_{2,1})} = \frac{(p_2 - p_1) \mathbb{I} \lbrace p_2 \geq p_1 \rbrace}{(1-\delta)\varepsilon'/ \delta}
        \end{equation*}
        and
        \begin{equation*}
             \frac{d Q_{2,1}}{d(P_{1,1} + P_{2,1})} = \frac{(p_1 - p_2) \mathbb{I} \lbrace p_1 \geq p_2 \rbrace}{\text{TV}(P_{1,1}, P_{2,1})} = \frac{(p_2 - p_1) \mathbb{I} \lbrace p_2 \geq p_1 \rbrace}{(1-\delta)\varepsilon'/ \delta}           
        \end{equation*}
        which are probability measures.
        \begin{proof}
            First, notice that:
            \begin{equation*}
                \int (p_2 - p_1) \mathbb{I} \lbrace p_2 \geq p_1 \rbrace = \int (p_1 - p2)\mathbb{I} \lbrace p_1 \geq p_2 \rbrace
            \end{equation*}
            since their difference is 0 and both are positive. Notice also that:
            \begin{equation*}
                \int (p_2 - p_1) \mathbb{I}\lbrace p_2 \geq p_1 \rbrace + \int (p_1 - p_2) \mathbb{I} \lbrace p_1 \geq p_2 \rbrace = 2 \text{TV} (P_{1,1}, P_{2,1})
            \end{equation*}
            Then we have that:
            \begin{equation*}
                \int (p_2 - p_1) \mathbb{I}\lbrace p_2 \geq p_1 \rbrace = \int (p_1 - p_2) \mathbb{I} \lbrace p_1 \geq p_2 \rbrace = \text{TV}(P_{1,1}, P_{2,1})
            \end{equation*}
            and
            \begin{equation*}
                \int \frac{d Q_{1,1}}{d(P_{1,1} + P_{2,1})}d(P_{1,1} + P_{2,1}) = 1
            \end{equation*}
            And the same goes for $Q_{2,1}$.
        \end{proof}
        
        We will now show that the contaminated measures:
        \begin{equation*}
            \tilde{P}_{1,1} = \delta P_{1,1} + (1-\delta)\varepsilon' Q_{1,1} \qquad \text{and} \qquad \tilde{P}_{2,1} = \delta P_{2,1} + (1-\delta)\varepsilon' Q_{2,1}
        \end{equation*}
        are in fact the same.
        \begin{proof} A simple computation gives that:
            \begin{equation*}
            \begin{split}
                \frac{d\tilde{P}_{1,1}}{d(P_{1,1} + P_{2,1})} & =  \delta p_1 + (1-\delta)\varepsilon' \frac{(p_2 - p_1) \mathbb{I} \lbrace p_2 \geq p_1 \rbrace}{\text{TV}(P_{1,1}, P_{2,1})}\\
                & = \delta \left(p_1 + (p_2 - p_1)\mathbb{I}\lbrace p_2 \geq p_1 \rbrace\right) \\
                & = \delta \left(p_2 + (p_1 - p_2) \mathbb{I}\lbrace p_1 \geq p_2 \rbrace\right) \\
                & = \delta p_2 + \frac{(p_2 - p_1) \mathbb{I} \lbrace p_2 \geq p_1 \rbrace}{(1-\delta)\varepsilon'/ \delta}\\
                & = \frac{d\tilde{P}_{2,1}}{d(P_{1,1} + P_{2,1})}
            \end{split}
            \end{equation*}
        \end{proof}

        Finally, notice that the contamination isn't exactly the one we are interested in. However, we can prove by adapting the proof of \citep[Lemma 7.2]{chenRobustCovarianceScatter2017} that:
        \begin{equation*}
            \lbrace \delta P_{1,1} + (1-\delta)\varepsilon' Q : Q\rbrace \subset \lbrace \delta P_{1,1} + (1-\delta)\varepsilon Q : Q\rbrace
        \end{equation*}
        \begin{proof}
            Let $p \in \lbrace \delta P_{1,1} + (1-\delta)\varepsilon' Q : Q\rbrace$ and $Q$ the contamination leading to $p$. Then, by setting $Q' = \frac{\varepsilon}{\varepsilon'} Q$ we have:
            \begin{equation*}
                \delta P_{1,1} + (1-\delta)\varepsilon' Q = \delta P_{1,1} + (1-\delta)\varepsilon Q'
            \end{equation*}
            Which proves the inclusion.
        \end{proof}

\section{Proofs of technical results}
\label{sec:other_proofs}
 
        \subsection{Proof of the correction formula  \eqref{eqn:mvcorrection}}
        \label{proof:formula_contaminated}
            
        Let $X$ be a zero mean random vector of $\mathbb{R}^p$ admitting covariance matrix $\Sigma$. Let $\xi$ be a zero mean random vector, independent from $X$, with diagonal covariance matrix $\Lambda$. Let $(d_j)_{1\leq j\leq p}$ and $(e_{j})_{1\leq j\leq p}$ sequences of Bernoulli random variables of probability respectively $\delta$ and $\varepsilon(1-\delta)$, independent from both $X$ and $\xi$ and such that $1 \leq j \leq p, d_je_j = 0$. Then, let $Y_i^{(j)} = d_j \odot X^{(j)} +e_j\odot \xi^{(j)}$. We have that: 
        \begin{equation}
            (Y \otimes Y)_{jk} = \begin{cases}
                d_j\left(X^{(j)}\right)^2 + e_j\left(\xi^{(j)}\right)^2  &\text{ if $j = k$}\\
                d_j d_k X^{(j)}X^{(k)} + d_j e_k X^{(j)}\xi^{(k)} + e_j d_k \xi^{(j)}X^{(k)}+ e_j e_k \xi^{(j)}\xi^{(k)} &\text{ otherwise}
            \end{cases} 
        \end{equation}
            
        This means that we have, by independence of the $X^{(j)}$ and the $\xi^{(j)}$, and by independence of the $\xi^{(j)}$ with each other:
        \begin{equation}
            \Sigma^Y_{jk} = \expectation \left( Y \otimes Y \right)_{jk} = \begin{cases}
                \delta \Sigma_{jj} + \varepsilon(1-\delta) \Lambda_{jj} & \text{ if $j = k$}\\
                \delta^2 \Sigma_{jk} & \text{ otherwise}
            \end{cases}
        \end{equation}
        Thus:
        \begin{equation}
            \Sigma_{jk} = \begin{cases}
                \delta^{-1} \left(\Sigma^Y_{jj} - \varepsilon(1-\delta) \Lambda_{jj}\right) & \text{ if $j = k$}\\
                \delta^{-2} \Sigma^Y_{jk}& \text{ otherwise}
            \end{cases}
        \end{equation}
        Which in turn means that:
        \begin{equation}
            \Sigma = (\delta^{-1} - \delta^{-2}) \diag(\Sigma^Y) + \delta^{-2} \Sigma^Y + \frac{\varepsilon(1-\delta)}{\delta} \Lambda
        \end{equation}
        This gives the general correction formula with independent contamination. For the missing values correction, simply set $\Lambda= \bm{0}$ the $p\times p$ zero matrix.
        
        \subsection{Bounds on the determinant of in equation \ref{eqn:pinsker}}
        \label{proof:determinant}
        
            Theorem 13 of  \cite{thompsonPrincipalSubmatricesNormal1966} states that, for any matrix $A$ of size $p$ with eigenvalues $\lambda_1, \dots \lambda_s$, each with multiplicity $\mu_1, \dots \mu_s$ such that $\sum_{i=1}^s \mu_i = p$, then any principal submatrix $A(j\vert j)$, that is, a matrix created by removing line $j$ and column $j$ from $A$, has eigenvalues $\lambda_i$ with multiplicity $\max (0, \mu_i - 1)$. The remaining eigenvalues have values between $\min_i \lambda_i$ and $\max_i \lambda_i$. 
            
            In our case, the matrix $\Sigma_j$ has only two eigenvalues: $1+\gamma$ and $1$, with multiplicity $r$ and $p-r$ respectively. One will easily find by recurrence on the number of deleted dimensions, which is $p-p_i$ with $p_i = \sum_{j=1}^p d_{i,j}$, that:
            \begin{equation}
            \det \Sigma_j^{(d_i)} = (1 + \gamma)^{\max (0, r-p+p_i)} \prod_{k=1}^{p-p_i} \lambda_k
            \end{equation}
            where $\forall k \in \lbrace 1, p_i \rbrace$, $1 \leq \lambda_k \leq 1+\gamma$.
            
            This means, in particular, that:
            \begin{equation}
                (1+\gamma)^{\max (0, r-p+p_i)} \leq \det \Sigma_j^{(p_i)} \leq (1+\gamma)^{\min (r, p_i)} 
            \end{equation}
            Now, let us demonstrate the statement in equation \ref{eqn:pinsker}. We have $\Sigma_1$ and $\Sigma_j$ having the same eigenvalues $1+\gamma$ and $1$ with multiplicity respectively $r$ and $p-r$. Let $p_i = \sum_{k=1}^p d_{i,k}$ be the number of remaining components after applying the boolean filter $d_i$ (thus there are $p-p_i$ deleted components). Since $\Sigma_1$ is diagonal, we know that $\Sigma_1^{(d_i)}$ will also have eigenvalues $1+\gamma$ and $1$, with multiplicity $a_i$ and $b_i$ respectively, where $a_i \sim \mathcal{B}(r, \delta)$ and $b_i \sim \mathcal{B}(p-r, \delta)$ where $\mathcal{B}$ is the binomial distribution.
            
            Then, using the lower bound we just demonstrated, we get that:
            \begin{equation}
                \begin{split}
                - \expectation_d \log \left(\det \left(\Sigma_1^{(d_i)-1} \Sigma_j^{(d_i)} \right) \right) & = \expectation_d a_i \log(1+\gamma) + b_i \log(1) - \log \left( \det \left(\Sigma_j^{(d_i)} \right)\right)\\
                & \leq \expectation_d a_i \log(1+\gamma) -  \max (0, r-p+p_i)\log(1+\gamma) \\
                & \leq \left( r\delta + \min (0, p-p_i-r) \right) \log (1+\gamma)\\
                & \leq r \delta \log(1+\gamma)
                \end{split}
            \end{equation}
            In particular, we know that $\gamma > 0$, so $\log(1+\gamma) \leq \gamma$ and
            \begin{equation}
                - \expectation_d \log \left(\det \left(\Sigma_1^{(d_i)-1} \Sigma_j^{(d_i)} \right) \right) \leq r \delta \gamma \leq  a \, r  \sqrt{p/n}.
            \end{equation}

        \subsection{Behaviour of the $Q_i$ with regard to matrix multiplication}

        We know that $Q_i Q_i^* = I_{d_i}$. Furthermore, $Q_i^* Q_i = I^{(J_i)}_p$, where $I^{(j_i)}_p$ is the diagonal matrix where the $j$th diagonal term is $1$ if only if $j \in J_i$, and $0$ otherwise.

        Finally, notice that in the general case, $Q_i A Q_i^* Q_i B Q_i^* \neq Q_i A B Q_i^*$, except when either $A$ or $B$ is diagonal. Indeed, for $k,l \in \lbrace 1, p \rbrace$:
        \begin{equation}
            \left(Q_i A Q_i^* Q_i B Q_i^*\right)_{kl} = \sum_{m=1}^p A_{km}B_{ml} \mathbb{I}_{k \in J_i} \mathbb{I}_{l \in J_i} \mathbb{I}_{m \in J_i}
        \end{equation}
        Which, if $A$ is diagonal, simply gives:
        \begin{equation}
        \begin{split}
            \left(Q_i A Q_i^* Q_i B Q_i^*\right)_{kl} & = A_{kk}B_{kl} \mathbb{I}_{k \in J_i} \mathbb{I}_{l \in J_i} = \left(Q_i A B Q_i^*\right)_{kl}
        \end{split}
        \end{equation}        
        
        \subsection{Proof of the upper bound of the frobenius norm with missing values}
        \label{proof:frobenius}
        
            Let $P\in \real^{p\times p}$ be any matrix, then, using the fact that the $d_i$ are boolean vectors:
            \begin{equation}
            \begin{split}
            \expectation_d\norm{(d_i \otimes d_i) \odot P}_F^2 & = \expectation_d \trace{\left((d_i \otimes d_i) \odot P\right)^\top \left((d_i \otimes d_i) \odot P\right)}\\
            & = \expectation_d \sum_{k=1}^p \sum_{l = 1}^p d_i^kd_i^l P_{kl}^2\\
            & = \sum_{k = 1}^p \left( \delta P_{kk} + \sum_{\substack{l=1\\l\neq k}}^p \delta^2 P_{kl}^2\right) \leq \delta\norm{P}_F^2
            \end{split}
            \end{equation}
            
\section{Tables}
\label{app:tables}

\begin{table}[ht]
        \centering
        \caption{We consider the cell-wise contamination model (\eqref{eqn:contaminated}) with a Gaussian contamination of high intensity, $\varepsilon = 1$ and for several values of $\delta$ in a grid. For each $\delta$, we average the proportion
of real data $\hat{\delta}$ and contaminated data $\hat{\varepsilon}$ after filtering over 20 repetitions. Values are displayed in
percentages ($\hat{\delta}$ must be high, $\hat{\varepsilon}$ low, both are expressed in percentages).}
        \vskip 0.15in
        \begin{small}
        \begin{sc}
        \scalebox{0.85}{
        \begin{tabular}{c || c c | c c | c c | c c| c c | c c}
        \toprule
             Contamination &  \multicolumn{4}{c|}{Tail cut} & \multicolumn{4}{c|}{DDC $99\%$} & \multicolumn{4}{c}{DDC $90\%$}\\
             rate & $\hat{\delta}$ & std &  $\hat{\varepsilon}$ & std & $\hat{\delta}$ & std &  $\hat{\varepsilon}$ & std & $\hat{\delta}$ & std & $\hat{\varepsilon}$ & std\\
             \midrule
            0.1$\%$ & 99.6 & 0.025 & 0.034 &  0.003 & 99.0 &  0.033 & 0.055 & 0.003 & 94.8 & 0.091 & 0.053 & 0.003 \\
            1$\%$ & 98.8 & 0.025 & 0.372 & 0.022 & 98.2 & 0.040 & 0.597 & 0.015 & 94.1 & 0.058 & 0.565 &  0.016 \\
            5$\%$ & 94.9 & 0.011 & 1.87 & 0.157 & 94.5 & 0.035 & 3.01 & 0.055 & 91.1 & 0.090 & 2.84 & 0.046 \\
            10$\%$ & 89.9 & 0.008 & 3.99 & 0.277 & 89.6 & 0.017 & 6.19 & 0.093 & 87.1 & 0.052 & 5.80 & 0.064 \\
            20$\%$ & 80.0 & 0.003 & 9.69 & 0.239 & 79.7 & 0.028 & 13.8 & 0.113 & 78.4 & 0.072 & 12.6 & 0.104 \\
            30$\%$ & 70.0 &  0.000 & 17.1 & 0.705 & 70.0 &  0.001 & 22.1 &  0.387 & 69.6 &  0.038 & 19.7 & 0.275 \\
             \bottomrule
        \end{tabular}}
        \end{sc}
        \end{small}
        \vskip -0.1in
        \label{tab:conta_gauss}
    \end{table}

\begin{table}[ht]
        \centering
        \caption{Same table on the Abalone dataset, contaminated with a Dirac contamination.}
        \vskip 0.15in
        \begin{small}
        \begin{sc}
        \scalebox{0.85}{
        \begin{tabular}{c || c c | c c | c c | c c| c c | c c}
        \toprule
             Contamination &  \multicolumn{4}{c|}{Tail cut} & \multicolumn{4}{c|}{DDC $99\%$} & \multicolumn{4}{c}{DDC $90\%$}\\
             rate & $\hat{\delta}$ & std &  $\hat{\varepsilon}$ & std & $\hat{\delta}$ & std &  $\hat{\varepsilon}$ & std & $\hat{\delta}$ & std & $\hat{\varepsilon}$ & std\\
             \midrule
             0.1$\%$ & 69.5 &  0.001 & 0.000 & 0.000 & 98.0 & 0.010 & 0.000 & 0.000 & 93.2 & 0.020 & 0.000 &  0.000 \\
             1$\%$ & 68.9 & 0.005 & 0.000 & 0.000 & 97.2 & 0.023 &0.000 & 0.000 & 92.6 &  0.039 & 0.000 & 0.000 \\
             5$\%$ & 66.2 & 0.034 & 0.000 & 0.000 & 93.6 & 0.043 & 0.000 & 0.000 & 89.8 & 0.083 & 0.000 & 0.000 \\
             10$\%$ & 62.8 & 0.016 & 0.000 & 0.000 & 89.0 & 0.034 & 0.000 & 0.000 & 86.0 & 0.045 & 0.000 & 0.000 \\
             20$\%$ & 56.0 & 0.002 & 6.00 & 0.000 & 79.9 & 0.070 & 0.138 & 0.163 & 79.6 & 0.355 &  0.001 & 0.003 \\
             30$\%$ & 49.0 & 0.000 & 9.00 & 0.000 & 70.0 & 0.000 & 29.5 & 0.036 & 70.0 & 0.000 & 24.2 & 0.127 \\
             \bottomrule
        \end{tabular}}
        \end{sc}
        \end{small}
        \vskip -0.1in
        \label{tab:conta_dirac_abalone}
    \end{table}

\begin{table}[t]
        \centering
        \caption{Same table on the Abalone dataset, contaminated with a Gauss contamination.}
        \vskip 0.15in
        \begin{small}
        \begin{sc}
        \scalebox{0.85}{
        \begin{tabular}{c || c c | c c | c c | c c| c c | c c}
        \toprule
             Contamination &  \multicolumn{4}{c|}{Tail cut} & \multicolumn{4}{c|}{DDC $99\%$} & \multicolumn{4}{c}{DDC $90\%$}\\
             rate & $\hat{\delta}$ & std &  $\hat{\varepsilon}$ & std & $\hat{\delta}$ & std &  $\hat{\varepsilon}$ & std & $\hat{\delta}$ & std & $\hat{\varepsilon}$ & std\\
             \midrule
             0.1$\%$ & 69.5 & 0.001 & 0.016 & 0.010 & 98.0 & 0.013 & 0.059 & 0.009 & 93.2 & 0.019 &  0.056 & 0.009\\
             1$\%$ &68.9 & 0.004 & 0.162 & 0.029 & 97.7 & 0.044 & 0.570 & 0.040 & 92.6 & 0.075 & 0.545 & 0.042\\
             5$\%$ & 66.2 & 0.028 & 0.852 & 0.055 & 93.5 & 0.058 & 2.86 & 0.045 & 89.8 &  0.119 & 2.73 & 0.050\\
             10$\%$ & 62.8 & 0.012 & 1.80 & 0.072 & 88.8 & 0.047 & 5.84 &  0.089 & 85.9 & 0.111 & 5.56 & 0.100 \\
             20$\%$ & 55.9 & 0.008 & 3.95 & 0.088 & 79.6 & 0.044 & 12.5 & 0.098 & 77.7 & 0.123 & 11.6 & 0.103 \\
             30$\%$ & 49.0 & 0.003 & 6.62 & 0.093 & 68.0 & 0.553 & 21.3 & 0.892 & 66.8 & 0.746 & 19.5 & 0.662 \\
             \bottomrule
        \end{tabular}}
        \end{sc}
        \end{small}
        \vskip -0.1in
        \label{tab:conta_gauss_abalone}
    \end{table}

\begin{figure}
    \begin{minipage}[t]{0.45\textwidth}
        \centering
    \includegraphics[width=\textwidth]{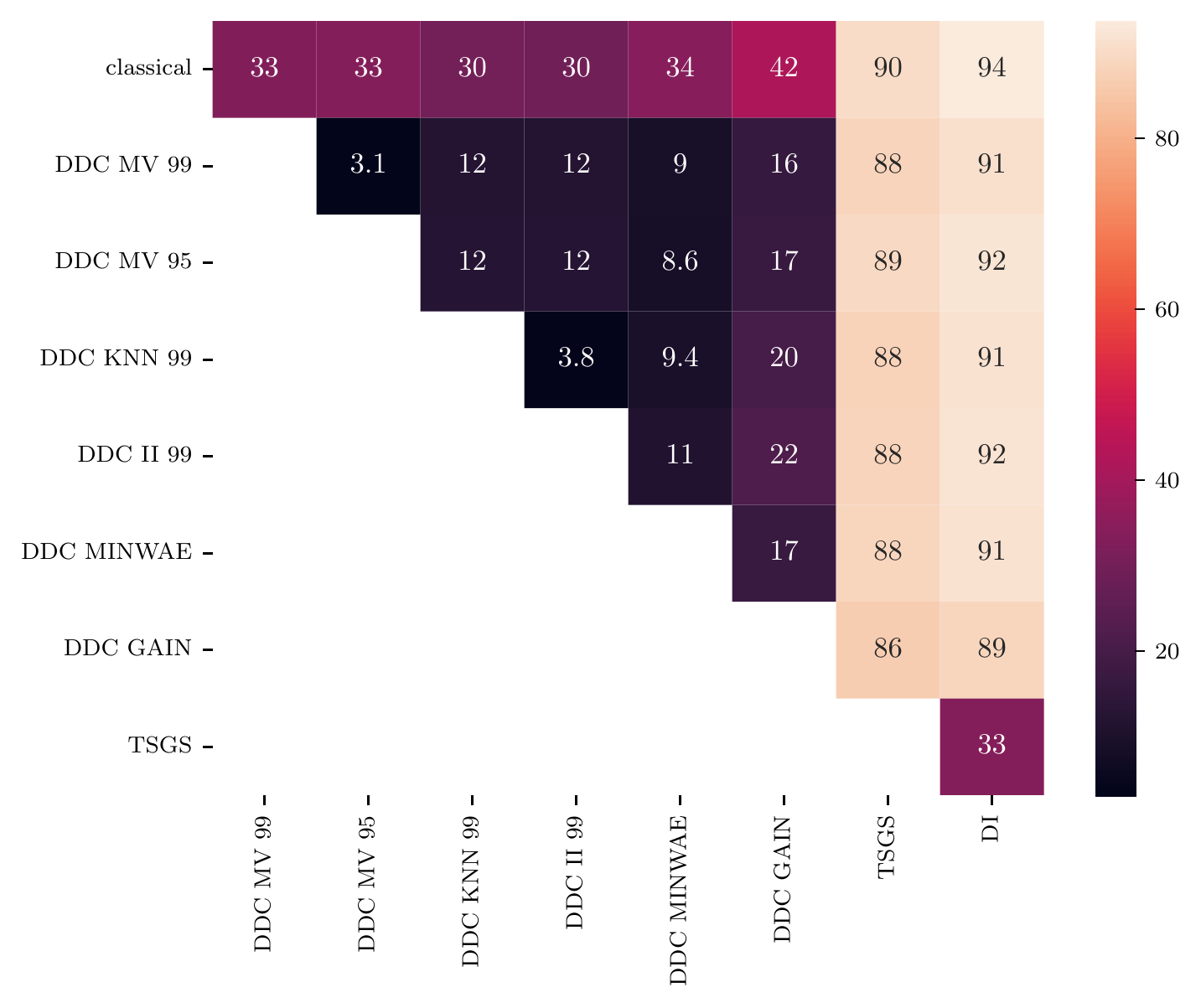}
    \caption{Relative spectral difference (in percentages) between estimated covariance matrices of the $11$ features of the R camera dataset.}
    \label{fig:cameras}
    \end{minipage}
    \hfill
    \begin{minipage}[t]{0.45\textwidth}
        \centering
     \includegraphics[width=\textwidth]{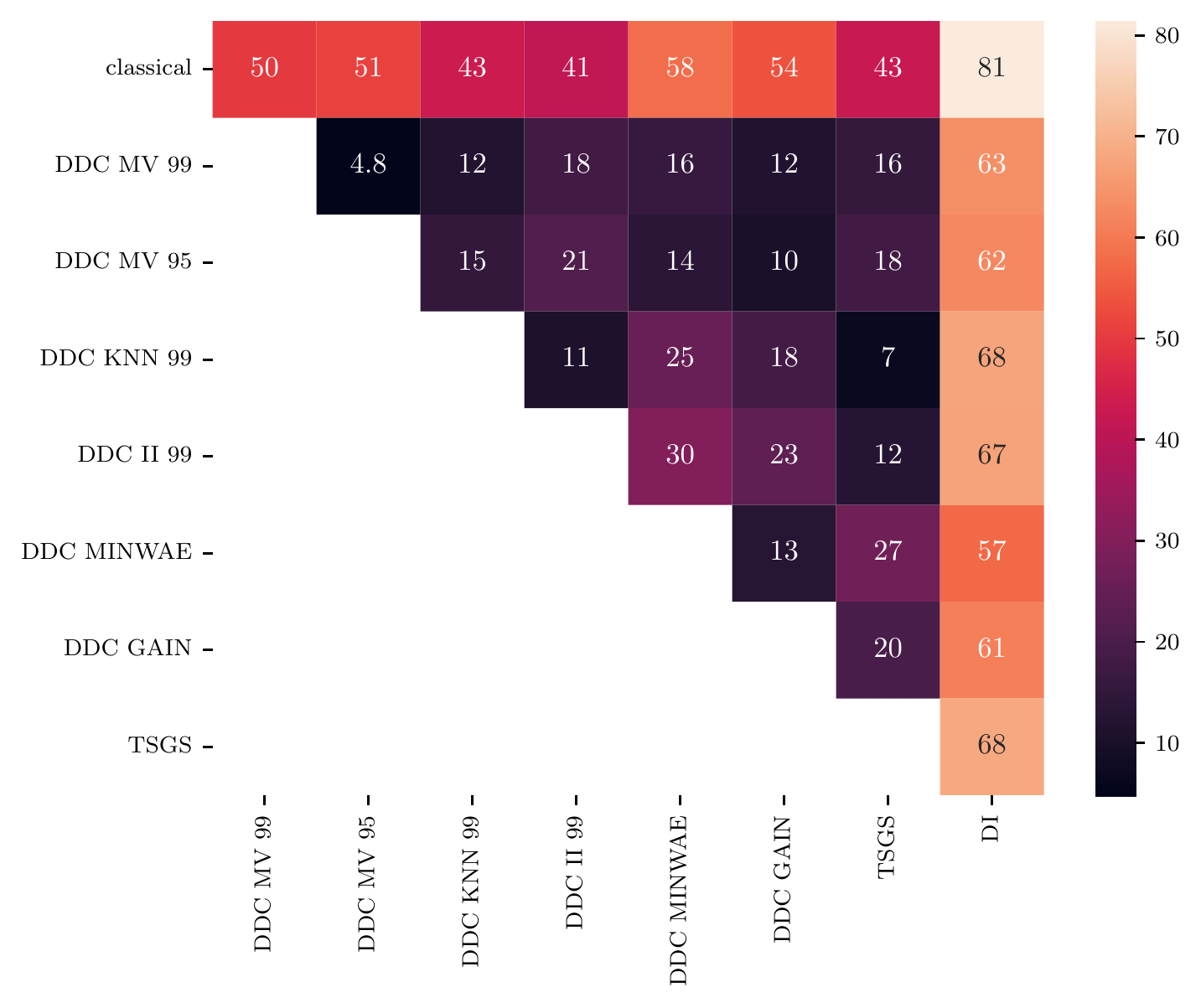}
                \caption{Relative spectral difference (in percentages) between estimated covariance matrices of the 30 features of sklearn's Breast Cancer. \texttt{DI} disagrees with every other procedures, casting some doubt on the reliability of its estimate.}
    \label{fig:breast_cancer}
    \end{minipage}
\end{figure}

\begin{figure}
    \begin{minipage}[t]{0.45\textwidth}
        \centering
    \includegraphics[width=\textwidth]{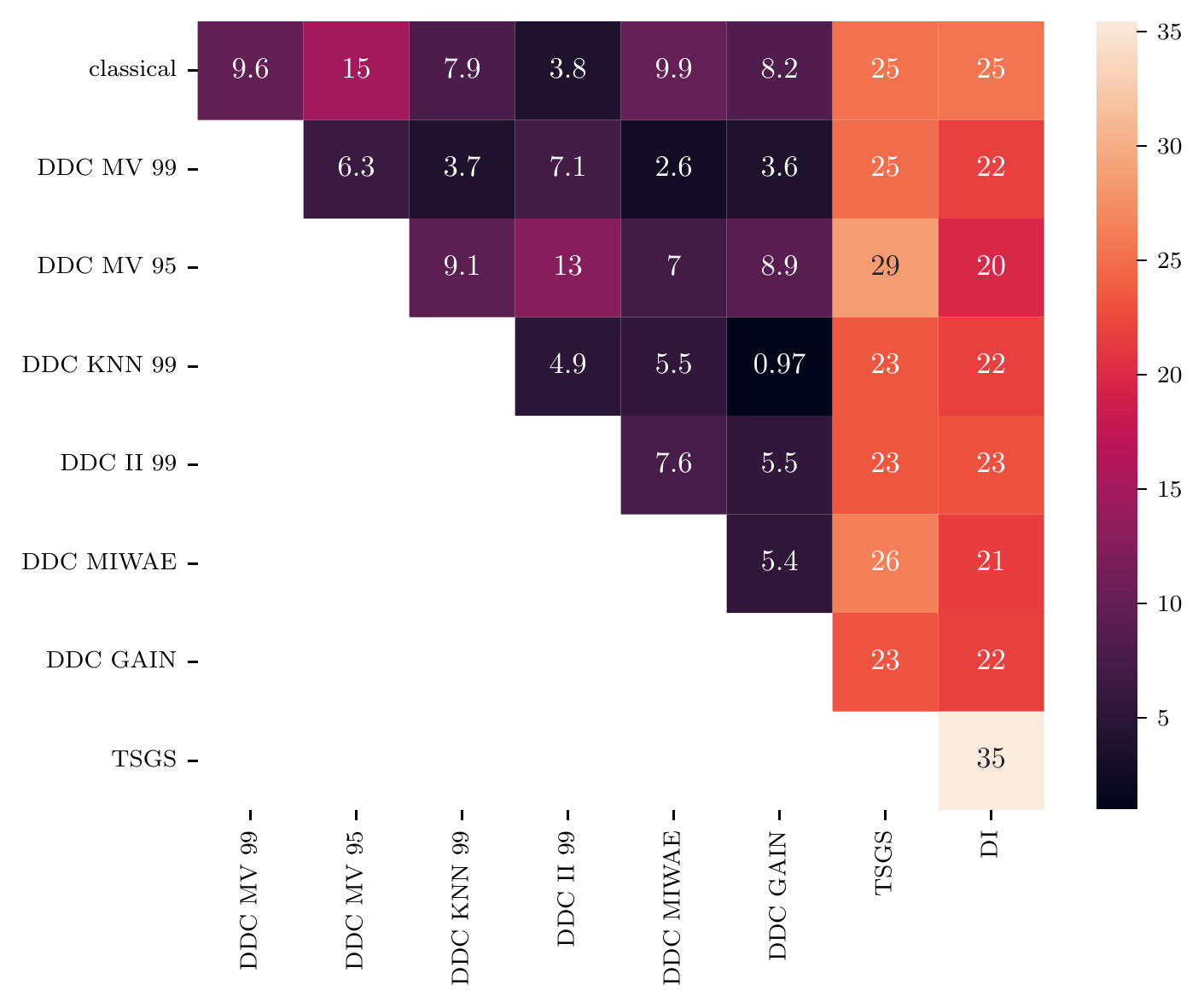}
    \caption{Relative spectral difference (in percentages) between estimated covariance matrices of the 11 features of the Woolridge Barium dataset.}
    \label{fig:barium}
    \end{minipage}
    \hfill
    \begin{minipage}[t]{0.45\textwidth}
        \centering
        \includegraphics[width=\textwidth]{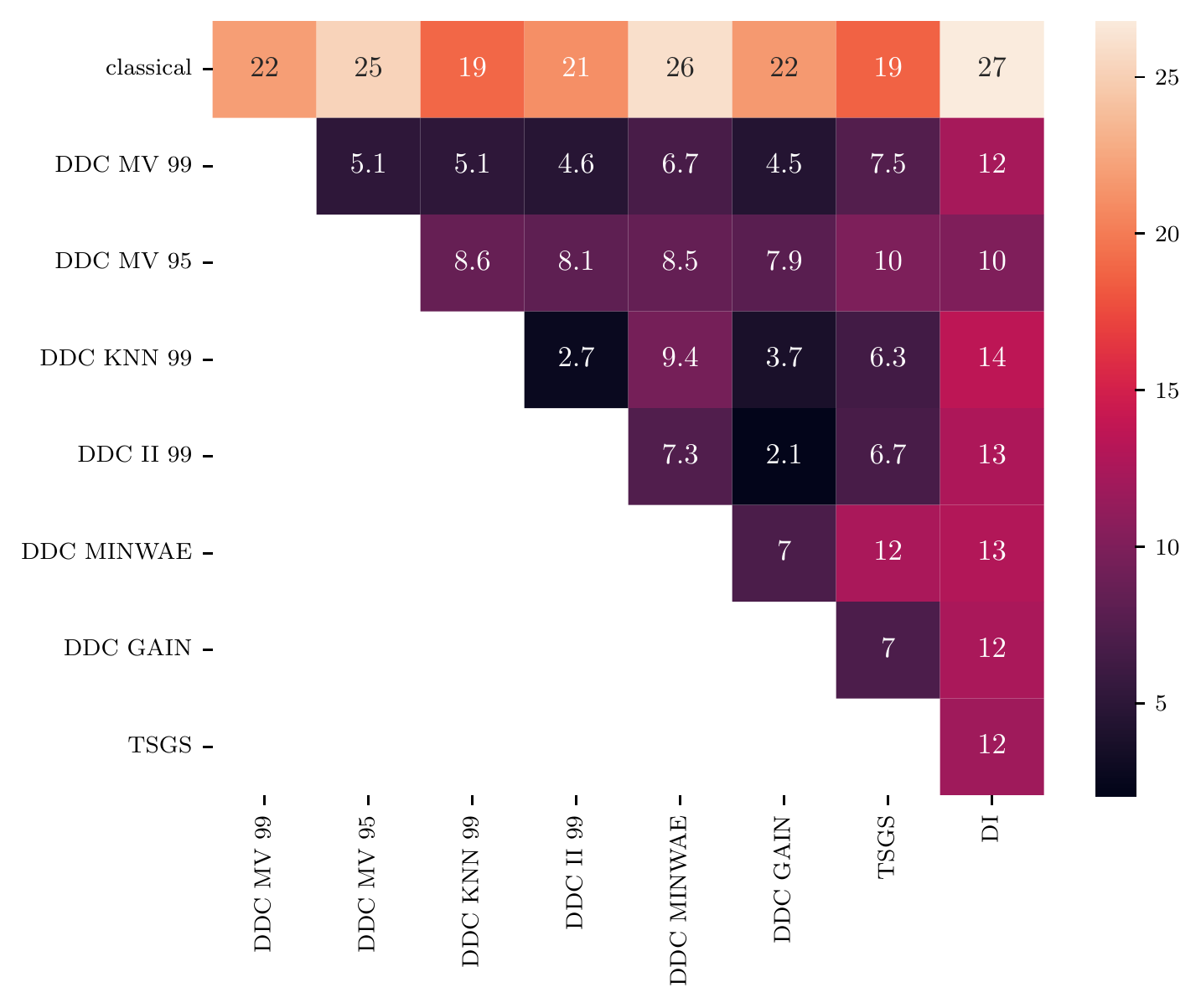}
        \caption{Relative spectral difference (in percentages) between estimated covariance matrices of the 13 features of sklearn's Wine dataset.}
        \label{fig:wine}
    \end{minipage}
\end{figure}
\begin{figure}
    \centering
 \includegraphics[width=0.5\textwidth]{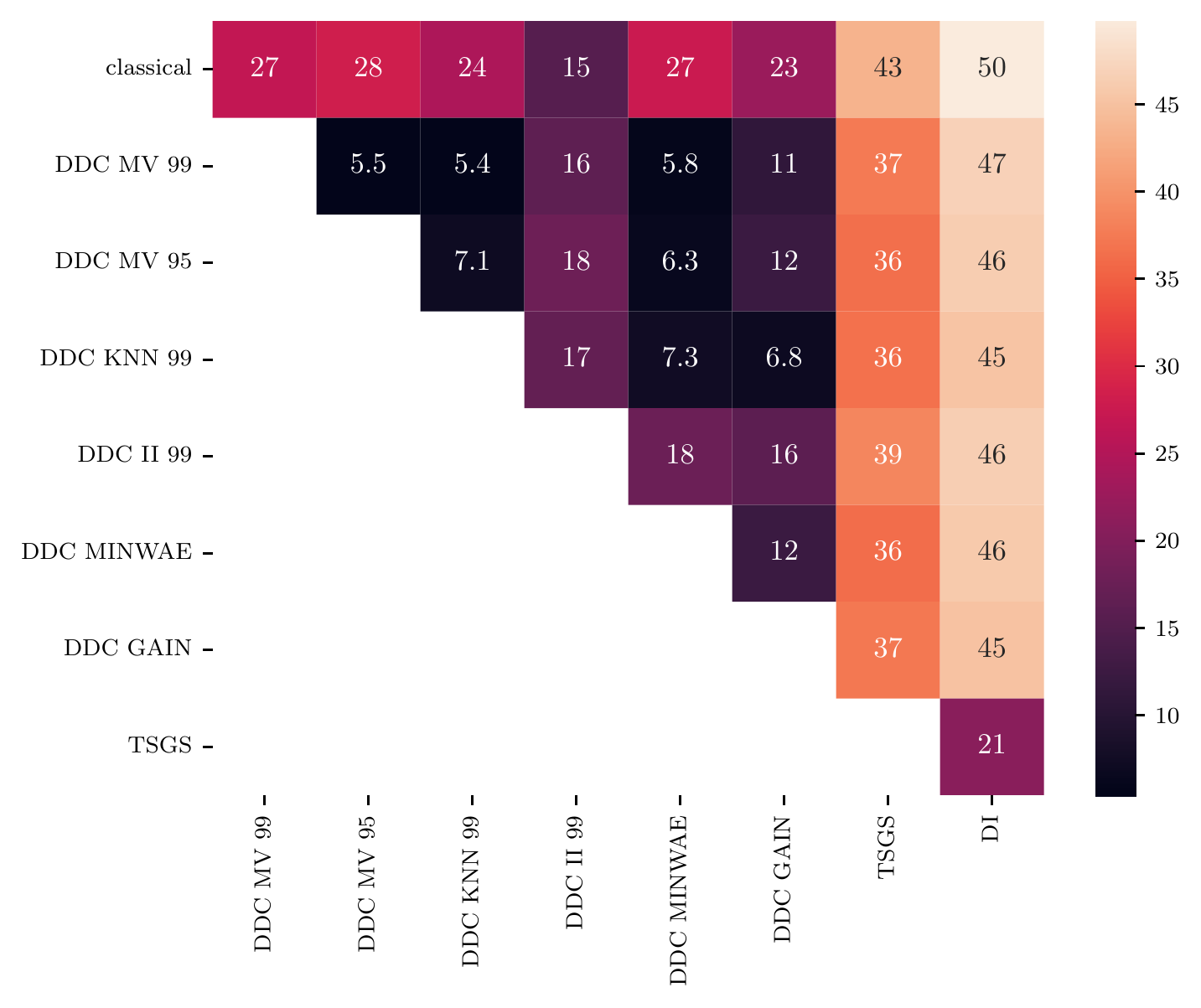}
    \caption{Relative spectral difference (in percentages) between estimated covariance matrices of the 13 features of Woolridge's INTDEF dataset.}
    \label{fig:intdef}
\end{figure}
    
\end{document}